\documentclass[11pt]{article}
\usepackage[margin = 1in]{geometry}
\usepackage{dsfont}
\usepackage{float}
\usepackage{amsmath,amssymb,amsfonts,amsthm}
\usepackage{url,graphicx,tabularx,array,appendix,geometry,multirow,array}
\usepackage{hhline}
\usepackage{natbib}
\usepackage{xy}\xyoption{all} \xyoption{poly} \xyoption{knot}
\usepackage{titlesec}
\usepackage[usenames,dvipsnames]{color}
\usepackage{enumerate}
\usepackage{latexsym}

\titleformat{\section}{\sf\large}{\thesection.}{1em}{\filcenter}
\titleformat{\subsection}{\sf}{\thesubsection}{1em}{}
\titleformat{\subsubsection}{\sf}{\thesubsubsection}{1em}{}
\newcolumntype{L}[1]{>{\raggedright\let\newline\\\arraybackslash\hspace{0pt}}m{#1}}
\newcolumntype{C}[1]{>{\centering\let\newline\\\arraybackslash\hspace{0pt}}m{#1}}
\newcolumntype{R}[1]{>{\raggedleft\let\newline\\\arraybackslash\hspace{0pt}}m{#1}}

\numberwithin{equation}{section}

\def\bbR{\mathbb{R}}

\def\bbN{\mathbb{N}}

\newtheorem{theorem}{Theorem}[section]

\newtheorem{lemma}[theorem]{Lemma}
\newtheorem{Lemma}{\indent \bf Lemma}
\newtheorem{corollary}[theorem]{Corollary}

\newtheorem{assumption}{Assumption}

\newtheorem{remark}{Remark}[section]

\newtheorem{example}{Example}[section]
\makeatletter\def\theequation{\arabic{section}.\arabic{equation}}
\@addtoreset{equation}{section}\makeatother

\begin{document}

\title{\bf \Large Semiparametric Bernstein-von Mises Theorem: Second Order Studies}
\author{Yun Yang\thanks{Department of Statistical Science, Duke University, E-mail: yy84@stat.duke.edu}, \, Guang Cheng\thanks{Department of Statistics, Purdue University, West Lafayette, IN 47907, Email: chengg@purdue.edu. Research Sponsored by NSF CAREER Award DMS-1151692, DMS-1418042 and Simons Foundation 305266. Guang Cheng was on sabbatical at Princeton while part of this work was finalized; he would like to thank the Princeton ORFE department for its hospitality.}\, \,  and David B. Dunson\thanks{Department of Statistical Science, Duke University, Box 90251, NC 27708, Durham, USA, Email: dunson@stat.duke.edu.}}
\date{ \today}
\maketitle

\begin{abstract}
The major goal of this paper is to study the second order frequentist properties of the marginal posterior distribution of the parametric component in semiparametric Bayesian models, in particular, a second order semiparametric Bernstein-von Mises (BvM) Theorem. Our first contribution is to discover an interesting interference phenomenon between Bayesian estimation and frequentist inferential accuracy: more accurate Bayesian estimation on the nuisance function leads to higher frequentist inferential accuracy on the parametric component. As the second contribution, we propose a new class of dependent priors under which Bayesian inference procedures for the parametric component are not only efficient but also adaptive (w.r.t. the smoothness of nonparametric component) up to the second order frequentist validity. However, commonly used independent priors may even fail to produce a desirable root-n contraction rate for the parametric component in this adaptive case unless some stringent assumption is imposed. Three important classes of semiparametric models are examined, and extensive simulations are also provided.

\end{abstract}

\section{Introduction}

A semiparametric model is indexed by a Euclidean parameter of interest $\theta\in\Theta\subset\bbR^p$ and an infinite-dimensional nuisance function $\eta$ belonging to a Banach space $\mathcal H$. For example, in the Cox proportional hazards model, $\theta$ is a
regression covariate vector corresponding to the log hazard ratio, while $\eta$ is a cumulative hazard function. By introducing a joint prior $\Pi$ on the product space $\Theta\times\mathcal H$, we can make Bayesian inferences for the parameter of interest, e.g., credible set, through MCMC sampling from the marginal posterior distribution. The frequentist validity of these Bayesian procedures is known to be supported by the semiparametric Bernstein-von Mises (BvM) theorem (see \cite{shen2001, bickel2012,Castillo2012}), which states that the marginal posterior distribution of $\theta$ is asymptotically normal and satisfies frequentist criteria of semiparametric efficiency. More precisely, it is proven to converge (in total variation norm) to a Gaussian limit centered at a semiparametric efficient estimate, with covariance matrix equal to the inverse of the efficient Fisher information:
\begin{equation}\label{eq:semiBvM}
    \sup_A\big|\Pi(\theta\in A|X_1,\ldots,X_n)-N_p\big(\theta_0+n^{-1/2}\widetilde{\Delta}_n,(n\widetilde I_{\theta_0,\eta_0})^{-1}\big)(A)\big|
    \overset{P_{\theta_0, \eta_0}}{\longrightarrow}0,
\end{equation}
where $A$ is any measurable subset of $\Theta$ and
\begin{align}\label{eq:esf}
    \widetilde{\Delta}_n:=\frac{1}{\sqrt{n}}\sum_{i=1}^n\widetilde{I}_{\theta_0,\eta_0}^{-1}
    \widetilde{\ell}_{\theta_0,\eta_0}(X_i)\overset{P_{\theta_0,\eta_0}}{\rightsquigarrow}N_p(0, \widetilde{I}_{\theta_0,\eta_0}^{-1})
\end{align}
reflects a random fluctuation on the center of the posterior distribution. Here, $P_{\theta_0,\eta_0}$ denotes the true underlying distribution that generates the data $(X_1,\ldots,X_n)$, where $\theta_0$ and $\eta_0$ are true parameter values;  $``\overset{P}{\rightsquigarrow}"$ and $``\overset{P}{\rightarrow}"$ denote the weak convergence and convergence in probability $P$, respectively. In the expression displayed above, $\widetilde I_{\theta,\eta}$ ($\widetilde{\ell}_{\theta,\eta}$) represents the efficient Fisher information matrix (efficient score function) evaluated at $(\theta,\eta)$. Please see Section \ref{se:rset} for a brief review on semiparametric efficiency theory. We call (\ref{eq:esf}) as the first order version of semiparametric BvM theorem. The recent studies of BvM theorem in the nonparametric context can be found in \citet{Shang2014,Castillo2014}.

The major goal of this paper is to conduct second order studies of semiparametric BvM theorem by characterizing the decay rate of the remainder term in \eqref{eq:semiBvM}, which we name as the second order Bayesian efficiency. This efficiency consideration is crucial for us to understand the influence of the nonparametric prior on the semiparametric Bayesian inferential accuracy, and further provide guidance in choosing an appropriate nonparametric prior (or more generally, a joint prior $\Pi$). We remark that our second order result is radically different from those in \cite{cheng2008b, cheng2008, cheng2009} where the nuisance function is profiled out, and thus no nonparametric prior needs to be assigned. Therefore, as far as we are aware, our work is the first study on the second order semiparametric BvM theorem in a fully Bayesian framework.

Our main conclusion is that the second order efficiency in \eqref{eq:semiBvM} is of the order $O_{P_{\theta_0,\eta_0}}(\sqrt{n}\rho_n^2)$ (upto a logarithmic term), where $\rho_n$ refers to the posterior contraction rate of the nuisance function throughout the paper. For example, in the partially linear models, the posterior contraction rate achieves $\rho_n=n^{-\alpha/(2\alpha + d)}(\log n)^\gamma$ for some $\gamma>0$, which is known to be (almost) minimax optimal, when an appropriate Gaussian process (GP) prior is assigned to the $d$-dimensional nonparametric function of $\alpha$-smoothness. Our general result implies an interesting interference phenomenon between Bayesian estimation and frequentist inferential accuracy: more accurate Bayesian estimation of the nuisance function leads to higher frequentist inferential accuracy on the parametric part. For example, we show that the credible set for $\theta$ possesses a second order frequentist validity that is determined by $\rho_n$. Please see Section~\ref{se:Implication} for more discussions. Therefore, it is desirable to construct a nonparametric prior under which an optimal contraction rate can be achieved. For example, it is desirable to match the smoothness of the reproducing kernel Hilbert space (RKHS) induced by the assigned GP with that of true regression function. None of the aforementioned interesting conclusions can be inferred from the first order semiparametric BvM theorem. We also remark that our second order Bayesian efficiency is consistent with that derived in \cite{cheng2008b, cheng2008, cheng2009} (up to a logarithmic term) where no nonparametric prior is assigned. Note that \cite{cheng2008b, cheng2008, cheng2009} is not a fully Bayesian framework, and did not cover the adaptive case considered in this paper. In the end, we point out that the above second order results are derived only under two intuitively appealing conditions: one is on the posterior concentration; another is on the integrated local asymptotic normality \citep{bickel2012}. Interestingly, these two conditions (together with a set of sufficient conditions in Section~\ref{sec:assver}) are not stronger than those imposed in the literature for the first order result, e.g., \cite{bickel2012,Castillo2012}. On the contrary, we even relax a stringent root-$n$ convergence  condition in \cite{bickel2012} to a set of commonly used conditions; see Lemma~\ref{le:pcr2}.

We further apply our general theory to two classes of priors varying by whether $\theta$ and $\eta$ are dependent or not. Surprisingly, we find that the commonly used independent prior is not the best choice for the second order semiparametric BvM theorem in the sense it requires a slightly strong condition (A6), and might even break down for the first order consistency when the smoothness of the nonparametric function is unknown; see Section~\ref{se:ipac} for more explanations. This failure is mainly due to the existence of a semiparametric bias term defined in (\ref{sembia}); also see \citet{Rivoirard2012}. Interestingly, we show that the semiparametric bias can be easily eliminated through shifting the center of a nonparametric prior (by a $\theta$-dependent quantity), which naturally leads to a general class of nonparametric priors. This re-centering idea is rather different from, and perhaps easier to implement than, the prior under-smoothing procedure proposed in \cite{Castillo2012b}. Moreover, our dependent priors can be easily made adaptive with respect to the unknown smoothness of the nuisance function by re-centering a nonparametric {\em adaptive} prior.

The rest of the paper is organized as follows. Section~\ref{sec:pre} provides necessary background on the semiparametric efficiency theory, and describes several semiparametric models including partially linear models and Cox proportional hazards models. Our main theorem, together with the related Bayesian inference results, is presented in Section~\ref{sec:main}. The classes of independent and dependent priors are extensively discussed in Section~\ref{se:sop}. Section~\ref{se:egs} illustrates the applicability of our general theory in three examples. All the technical proofs are postponed to the Appendix.

\section{Preliminaries}\label{sec:pre}
In this section, we briefly review semiparametric efficiency theory \citep{bickel1998}, and describe several semiparametric models considered in the paper.

\subsection{Review on Semiparametric Efficiency}\label{se:rset}

An estimator $\widehat{\theta}_n$ is semiparametric efficient if it achieves the minimal asymptotic variance $V^\ast$ over all regular semiparametric estimators $\widetilde{\theta}_n$ that satisfy $\sqrt{n}(\widetilde{\theta}_n-\theta_0)\overset{P}{\rightsquigarrow} N_p(0,V)$ for some non-degenerate asymptotic variance $V$. It can be shown that the minimal $V^\ast$ exists and corresponds to the largest asymptotic variance over all the parametric submodels $\{P_{\theta,\eta(\theta)}:\theta\in\Theta\}$ with $\eta(\theta_0)=\eta_0$ \citep{bickel1998}. The submodel achieving $V^*$ is called the least favorable submodel, and denoted as $\{P_{\theta,\eta^\ast(\theta)}:\theta\in\Theta\}$, where $\eta^\ast(\theta)$ is the so-called least favorable curve. We define the semiparametric bias as
\begin{eqnarray}\label{sembia}
\Delta\eta(\theta)=\eta^\ast(\theta)-\eta^\ast(\theta_0)=\eta^\ast(\theta)-\eta_0,
\end{eqnarray}
which will be frequently mentioned hereafter. Let $\widetilde{\ell}_{\theta_0,\eta_0}$ be the score function of the least favorable submodel $\{P_{\theta,\eta^\ast(\theta)}:\theta\in\Theta\}$ at $\theta=\theta_0$. Hence, we have $V^\ast=\widetilde{I}_{\theta_0,\eta_0}^{-1}$, where $\widetilde{I}_{\theta_0,\eta_0}=E_0\widetilde{\ell}_{\theta_0,\eta_0}\widetilde{\ell}_{\theta_0,\eta_0}^T$. Note that $\widetilde\ell_{\theta_0,\eta_0}$ and $\widetilde I_{\theta_0,\eta_0}$ are also known as the efficient score function and efficient information matrix in the semiparametric literature. For simplicity, denote $P_{\theta_0,\eta_0}$, $\widetilde\ell_{\theta_0,\eta_0}$ and $\widetilde I_{\theta_0,\eta_0}$ as $P_0$, $\widetilde\ell_{0}$ and $\widetilde I_{0}$ from now on.

\citet{severini1992} discovered that $\eta^{\ast}(\theta)$ is essentially evaluated as the unique minimizer of Kullback-Leibler (KL) divergence in $\mathcal{H}$ with the parametric part $\theta$ being fixed, i.e.
\begin{align}\label{eq:KL}
    \eta^{\ast}(\theta)
    =\text{arg}\inf_{\eta\in\mathcal{H}}K(P_{0},P_{\theta,\eta}),
\end{align}
where $K(P,Q)=\int \log(dP/dQ)dP$ denotes the KL divergence between two measures $P$ and $Q$. In the Bayesian regime, the least favorable curve $\eta^\ast(\theta)$ can be understood as the function towards which the conditional posterior distribution of the nuisance parameter $\eta$ given $\theta$ contracts \citep{Kleijn2006}. Therefore, the posterior distribution of $(\theta,\eta)$ tends to concentrate around the true value $(\theta_0, \eta_0)$ under a well chosen prior $\Pi$. We use the following examples to illustrate the above concepts, in particular $\eta^*(\theta)$. 

\subsection{Generalized Partially Linear Model (GPLM)}

Suppose that the data $X_i=(U_i,V_i,Y_i)$, $i=1,\ldots,n$, are i.i.d.\! copies of $X=(U,V,Y)$, where $Y\in\bbR$ is the response variable and $T=(U,V)\in[0,1]^p\times [0,1]^d$ is the covariate variable. Consider a general class of semiparametric
regression models with the following partially linear structure:
\begin{align*}
    m_0(t)\equiv E_0(Y|T=t)=F(g_0(t)), \quad g_0(t)=\theta_0^Tu+\eta_0(v),\quad t=(u,v),
\end{align*}
where $F:\bbR\to\bbR$ is some known link function and $\eta_0$ is some unknown smooth function. Here the notation $E_0$ means the expectation under the true data generating probability measure $P_0=P_{\theta_0,\eta_0}$. The above class of semiparametric models is called generalized partially linear models (GPLM) \citep{Boente2006}. Interestingly, we can derive an explicit expression of the least favorable curve (see Lemma \ref{le:glfc}) when the log-likelihood is written in the following form\footnote{This form is also called as quasi-likelihood in \cite{Wed74}}: $\log p(y; m)=\int_{y}^m(y-s)/\mathcal V(s)ds$, where $\mathcal V(m_0(T))=Var(Y|T)$. We next apply Lemma \ref{le:glfc} to three concrete models.

\begin{example}[Partially linear models]\label{eg:plm}

Consider a partially linear regression model
\begin{align}\label{eq:plmf}
    Y=U^T\theta_0+\eta_0(V)+w,
\end{align}
where $w\sim N(0,1)$ is assumed to be independent of $(U,V)$ and $\eta_0$ belongs to a H\"older function class $C^{\alpha}([0,1]^d)$ with smoothness index $\alpha$. In this case, $F(t)=t$ and $\mathcal V(s)=1$. Based on Lemma \ref{le:glfc}, we obtain the least favorable curve as
\begin{align}\label{eq:lfc}
 \eta^*(\theta)(v)=\eta_0(v)-(\theta-\theta_0)^TE[U|V=v].
\end{align}
For identifiability, we assume that $E(U-E[U|V])^{\otimes 2}$ is invertible.
\end{example}

\begin{example}[Partially linear exponential models]
In the partially linear exponential model, the conditional density of $Y$ given $(U,V)$ is
\begin{align}
    p_0(y|u,v)=\lambda_0(u,v)\exp(-\lambda_0(u,v)y),\quad y>0,
\end{align}
with $\lambda_0(u,v)=\exp\{-(u^T\theta_0+\eta_0(v))\}$. In this case, $F(t)=e^t$ and $\mathcal V(s)=s^2$. Therefore, by Lemma \ref{le:glfc}, we have
\begin{align}
 \eta^*(\theta)(v)=\eta_0(v)-(\theta-\theta_0)^TE[U|V=v]+O(|\theta-\theta_0|^2).
\end{align}
\end{example}

\begin{example}[Partially linear logistic models]
In the partially linear logistic model, we observe binary $Y_i\in\{0,1\}$ and model the data as
\begin{align}
    \log\left\{\frac{P_0(Y=1|U,V)}{P_0(Y=0|U,V)}\right\}=U^T\theta_0+\eta_0(V).
\end{align}
In this case, $F(t)=e^t/(1+e^t)$ and $\mathcal V(s)=s(1-s)$. Again, Lemma \ref{le:glfc} implies the least favorable curve as
\begin{align}
 \eta^*(\theta)(v)=\eta_0(v)-(\theta-\theta_0)^T\frac{E[Uf_0(U,V)|V=v]}{E[f_0(U,V)|V=v]}+O(|\theta-\theta_0|^2),
\end{align}
where $f_0(u,v)=\exp(u^T\theta_0+\eta_0(v))/(1+\exp(u^T\theta_0+\eta_0(v)))^{2}$.
\end{example}

\subsection{Cox Proportional Hazards Model}\label{sec:cox}
Let $Z\in\bbR^p$ be a vector of covariates, $T$ the survival time that follows a Cox model, and $C$ a random observation time. The Cox model assumes that the conditional hazard function given $Z$ satisfies $P(T\in[t,t+dt]|T\geq t,Z)=\exp(\theta_0^T Z)\lambda_0(t)dt$, where $\lambda_0$ is an unknown baseline hazard function. Assume that $T$ is independent of $C$ given $Z$ and that there exists a real $\tau>0$ such that $P_0(T>\tau)>0$ and $P_0(C\geq \tau)=P_0(T=\tau)>0$. In the Cox model with current status data, the observed data are $n$ i.i.d. realizations of $X=(C,\delta,Z)$, where $\delta=I(T\leq C)$. The density of $X$ relative to the product of the marginal joint density of $(C,Z)$ and counting measure on $\{0,1\}$ is given by
\begin{align*}
    p_{\theta,\Lambda}(x)=\bigg(1-\exp\big(-e^{\theta^Tz}\Lambda(c)\big)\bigg)^{\delta}
    \bigg(\exp\big(-e^{\theta^Tz}\Lambda(c)\big)\bigg)^{1-\delta},
\end{align*}
where $\Lambda(c)=\int^{c}_0\lambda(t)dt$ is considered as the nuisance parameter. By the derivations in Section 25.11.1 of \cite{Van1998}, the least favorable curve is given by
\begin{align}\label{eq:lfdcox}
 \Lambda^\ast(\theta)(c)=\Lambda_0(c)-(\theta-\theta_0)^T\Lambda_0(c)
 \frac{E[ZQ^2_{\theta_0,\Lambda_0}(X)|C=c]}{E[Q^2_{\theta_0,\Lambda_0}(X)|C=c]}+O(|\theta-\theta_0|^2),
\end{align}
for the function $Q_{\theta,\Lambda}$ given by
\begin{align*}
    Q_{\theta,\Lambda}(x)=\exp(\theta^Tz)\bigg[\delta\frac{\exp(-e^{\theta^Tz}\Lambda(c))}{1-\exp(-e^{\theta^Tz}\Lambda(c))}
    -(1-\delta)\bigg].
\end{align*}

\section{Main Results}\label{sec:main}

\subsection{Second Order Semiparametric BvM Theorem}
For a general class of semiparametric models $\mathcal{P}=\{P_{\theta,\eta}:\theta\in\Theta,\eta\in\mathcal{H}\}$, we consider a joint prior distribution $\Pi$ over the product space $\Theta\times\mathcal{H}$ for the parameter pair $(\theta,\eta)$. In the sequel, we use notation $\Pi^{\theta}_{\mathcal{H}}(\eta)$ and $\Pi_{\Theta}(\theta)$ to denote the conditional prior distribution of $\eta$ given $\theta$ and the marginal prior distribution of $\theta$, respectively.

Our main theorem is based on two primary assumptions, which we will revisit in Section~\ref{sec:assver}. The first one is a convergence condition for $(\theta, \eta)$. It allows us to focus on the posterior mass in a suitable neighborhood of $(\theta_0, \eta_0)$: $\{(\theta,\eta):|\theta-\theta_0|\leq \epsilon_n,\eta\in \mathcal{H}_n\}$, where $|\cdot|$ denotes the Euclidean norm. Here, $\mathcal{H}_n$ is a sequence of subsets of the nuisance space $\mathcal{H}$ that satisfies $ \Pi(\eta\in\mathcal{H}_n|X_1,\ldots,X_n)\overset{P_0}{\longrightarrow}1$.
For example, $\mathcal{H}_n$ can be defined as $\{\eta: \|\eta-\eta_0\|_n \leq M\rho_n\}\cap\mathcal{F}_n^{\eta}$, where $\mathcal{F}_n^{\eta}$ is a sieve sequence for the nuisance parameter defined after Lemma \ref{le:pcr2} and $\|f\|_n=\big(n^{-1}\sum_{i=1}^nf^2(X_i)\big)^{1/2}$ is an empirical $L_2$-norm. Recall that $\rho_n$ denotes the contraction rate of marginal posterior distribution of $\eta$.

\begin{assumption}[Localization condition]\label{A.2} There exists a sequence $\epsilon_n\to 0$ satisfying $n\epsilon_n^2\to\infty$ and a sequence of subsets $\{\mathcal{H}_n\}\subset\mathcal H$, such that as $n\to\infty$,
$$
\Pi\big(|\theta-\theta_0|\leq \epsilon_n, \eta\in\mathcal{H}_n \big|X_1,\ldots,X_n\big)=1-O_{P_0}(\delta_n)
$$
for some $\delta_n\to 0$.
\end{assumption}
In a general setup, Lemma \ref{le:pcr2} in Section~\ref{sec:assver} provides a set of sufficient conditions for Assumption \ref{A.2} with $\epsilon_n=\rho_n$. Throughout the paper, we always choose $\epsilon_n$ to be $\rho_n$.

The second assumption extends the concept of local asymptotic normality (LAN) (required for the parametric BvM theorem in \cite{LeCam1953b}) to the semiparametric context. Denote the log-likelihood by $l_n(\theta,\eta)$. By Fubini's theorem, the marginal posterior for $\theta$ can be written as
\begin{equation}\label{semipos}
\begin{aligned}
    \Pi(\theta\in A|X_1,\ldots,X_n)=&\int_A\bigg\{\int_{\mathcal{H}}\exp\big(l_n(\theta,\eta)
    -l_n(\theta_0,\eta_0)\big)d\Pi_{H}^{\theta}(\eta)\bigg\}d\Pi_{\Theta}(\theta)\\
    &\bigg/\int_{\Theta}\bigg\{\int_{\mathcal{H}}\exp\big(l_n(\theta,\eta)
    -l_n(\theta_0,\eta_0)\big)d\Pi^{\theta}_{\mathcal{H}}(\eta)\bigg\}d\Pi_{\Theta}(\theta).
\end{aligned}
\end{equation}
Therefore, the integrated likelihood ratio $S_n(\theta)$ defined by the map
\begin{equation}\label{eq:inlr}
    S_n(\theta)=\int_{\mathcal{H}}\exp\big(l_n(\theta,\eta)
    -l_n(\theta_0,\eta_0)\big)d\Pi^{\theta}_{\mathcal{H}}(\eta),
\end{equation}
plays a similar role as the likelihood ratio in the parametric model. To prove the first order semiparametric BvM theorem, \cite{bickel2012} assume that for every random sequence $\{h_n\}$ of order $O_{P_0}(1)$,
\begin{align}\label{eq:ilana}
    \log\left\{\frac{S_n(\theta_0+n^{-1/2}h_n)}{S_n(\theta_0)}\right\}=h_n^T\widetilde{g}_n
    -\frac{1}{2}h_n^T\widetilde{I}_{0}h_n+o_{P_0}(1),
\end{align}
where $\widetilde{g}_n=(1/\sqrt{n})\sum_{i=1}^{n}\widetilde\ell_0(X_i)
\overset{P_0}{\rightsquigarrow} N_p(0, \widetilde I_{0})$.

Accompanied with (\ref{eq:ilana}), \cite{bickel2012} further require the marginal posterior of $\theta$ to converge at root-$n$ rate. In many cases, it may require significant effort to verify this parametric-rate condition. To avoid such a stringent assumption as well as keep track of the higher-order remainder, we introduce the notion of the {\em localized} integral likelihood ratio as follows:
\begin{equation}\label{eq:linlr}
\widetilde{S}_n(\theta)=\int_{\mathcal{H}_n}\exp\big(l_n(\theta,\eta)
    -l_n(\theta_0,\eta_0)\big)d\Pi^{\theta}_{\mathcal{H}}(\eta).
\end{equation}
The information in the localization sequence $\mathcal H_n$, e.g., $\|\eta-\eta_0\|_n\leq M\rho_n$ and $\eta\in\mathcal{F}_n^{\eta}$, will be utilized in the application of the maximal inequality \cite[Corollary 2.2.5]{Van1996} to provide a uniform bound. More importantly, when these conditions are combined with Assumption~\ref{A.2}, we no longer need to assume the root-$n$ marginal convergence rate for $\theta$.

\begin{assumption}[Second order integrated LAN]\label{A.3}
There exists a nondecreasing function $R_n(\cdot):\bbR\to\bbR$ satisfying $\sup_{t\in[n^{-1/2},\, M\epsilon_n]} R_n(t)/nt^2\to 0$ for each $M>0$ such that for every sequence $\theta_n$ satisfying $\theta_n=\theta_0+o_{P_{0}}(1)$,
\begin{equation}\label{eq:ilanb}
\begin{aligned} \log\frac{\widetilde{S}_n(\theta_n)}{\widetilde{S}_n(\theta_0)}=&\sqrt{n}(\theta_n-\theta_0)^T
    \widetilde{g}_n
    -\frac{n}{2}(\theta_n-\theta_0)^T\widetilde{I}_{0}
    (\theta_n-\theta_0)
    +O_{P_0}(R_n\big(|\theta_n-\theta_0|\big)).
\end{aligned}
\end{equation}
\end{assumption}
Note that \eqref{eq:ilanb} can be written in the form of \eqref{eq:ilana} by re-parameterizing $\theta_n$ as $\theta_0+n^{-1/2}h_n$. In the sequel, we name \eqref{eq:ilanb} as ILAN. A typical $R_n(t)$ is dominated by $\sqrt{n}t^2+\sqrt{n}\rho_n^2$; see the examples in Section~\ref{se:egs} and their proofs.

Now, we are ready to present the main theorem in this paper.
\begin{theorem}\label{thm:MAIN1}
We assume the prior for $\theta$ has a Lebesgue density that is continuous and strictly positive at $\theta_0$ and the efficient information matrix $\widetilde{I}_{0}$ is invertible. Suppose $X_1,\ldots,X_n$ are i.i.d.\! observations sampled from $P_0$. Under Assumptions \ref{A.2} and \ref{A.3}, we have
\begin{equation}\label{eq:sBvM}
    \sup_A\big|\Pi(\theta\in A|X_1,\ldots,X_n)-N_p\big(\theta_0+n^{-1/2}\widetilde{\Delta}_n,(n\widetilde{I}_{0})^{-1}\big)(A)\big|= O_{P_0}(S_n),
\end{equation}
where $A$ ranges over all measurable subsets of $\Theta$ and $S_n=R_n(n^{-1/2}\log n)+\delta_n$.
\end{theorem}
\noindent We remark that Theorem~\ref{thm:MAIN1} can be easily adapted to a non-asymptotic version (by invoking (\ref{inter1}) in Lemma~\ref{Lemma:MI}) if Assumptions~\ref{A.2} and~\ref{A.3} are stated in a non-asymptotic manner.


The $\log n$ term in $S_n$ is not essential and does not affect the polynomial order of $S_n$, which is our main interest. We need it in the proof such that the posterior probability of the event $\{\theta:Mn^{-1/2}\log n\leq|\theta-\theta_0|\leq \epsilon_n\}$ decays at a faster rate than $S_n$ for a sufficiently large $M$.

We comment that Assumption \ref{A.3} is implied by the following conditions: (A1) on the semiparametric model; and (A2) on the prior. Specifically, Lemma \ref{thm:2} in Section \ref{sec:assver} shows that $R_n(\cdot)=G_n(\cdot)+\widetilde{G}_n(\cdot)$, where $G_n$ and $\widetilde G_n$ are given in (A1) and (A2), respectively. Note that $\mathcal H_n$ in Conditions (A1) and (A2) is the same as that in Assumption~\ref{A.2}.

(A1) (Stochastic LAN) There exists an increasing function $G_n:\bbR\rightarrow[0,\infty)$, such that for every sequence $\{\theta_n\}$ satisfying $\theta_n=\theta_0+o_{P_0}(1)$,
\begin{equation}\label{eq:A1}
\begin{aligned}
&\sup_{\eta\in\mathcal{H}_n}\bigg|l_n\big(\theta_n,\eta+\Delta\eta(\theta_n)\big)-l_n(\theta_0,\eta)
-(\theta_n-\theta_0)^T\sum_{i=1}^n\widetilde{\ell}_{0}(X_i)\\
&\qquad\qquad+\frac{1}{2}n(\theta_n-\theta_0)^T\widetilde{I}_{0}(\theta_n-\theta_0)\bigg|
=O_{P_0}\big(G_n\big(|\theta_n-\theta_0|\big)\big).
\end{aligned}
\end{equation}

If we set $\eta=\eta_0$ in \eqref{eq:A1}, then we obtain the LAN for the least favorable submodel $l_n\big(\theta_n,\eta^{*}(\theta_n)\big)$. A typical $G_n(t)$ in \eqref{eq:A1} is dominated by $\sqrt{n}t^2+\sqrt{n}\rho_n^2$. For example, see the verification of (A1) in the proof of Theorem~\ref{thm:4a,3a}.

Condition (A2) characterizes the prior stability under a small perturbation in the likelihood function caused by $\Delta\eta(\theta_n)$ in the nuisance part.

(A2) (Prior stability under perturbation) There exists an increasing function $\widetilde{G}_n:\bbR\rightarrow[0,\infty)$, such that for any $\theta_n=\theta_0+o_{P_0}(1)$,
  \begin{align*}
   \frac{\int_{\mathcal{H}_n}\exp(l_n(\theta_0,\eta-\Delta\eta(\theta_n)))
    d\Pi_H^{\theta_n}(\eta)}{\int_{\mathcal{H}_n}\exp(l_n(\theta_0,\eta))
    d\Pi_H^{\theta_0}(\eta)}=1+O_{P_0}(\widetilde{G}_n(|\theta_n-\theta_0|)).
\end{align*}

Condition (A2) is crucial for proving the root-$n$ convergence rate of $\theta$ --- whose failure is typically caused by $\lim\inf_{n\to\infty}\widetilde{G}_n(n^{-1/2}\log n)>0$ (see the numerical study in Section \ref{se:sim}).  In fact, we call a nonparametric prior an unbiased one if $\lim_{n\to\infty}\widetilde{G}_n(n^{-1/2}\log n)=0$ since it corrects the semiparametric bias $\Delta\eta$ in (A2). In the special case \citep{bickel1982} that $\{P_{\theta,\eta_0}:\theta\in\Theta\}$ forms a least favorable submodel, i.e., $\Delta\eta\equiv0$, (A2) automatically holds when independent priors are assigned for $\theta$ and $\eta$. However, in the general case where $\Delta\eta\neq0$, we typically have $\Delta\eta(\theta_n)=O(|\theta_n-\theta_0|)$ (see (A3) in Section \ref{se:sop}) and that
$$
\exp\big\{l_n(\theta_0,\eta-\Delta\eta(\theta_n))-l_n(\theta_0,\eta)\big\}
=O_{P_0}(n|\theta_n-\theta_0|\rho_n)
$$
does not converge to zero. Therefore, under independent priors, (A2) cannot be implied by bounding the ratio between integrands in its denominator and numerator unless we are willing to impose additional conditions such as (A5) and (A6) in Section~\ref{se:sop1}.

\subsection{Second Order Bayesian Inference}\label{se:Implication}
In practice, we can employ an MCMC algorithm to efficiently draw a sequence of samples $\{\theta^{(l)}: l=1,\ldots,L\}$ from the marginal posterior distribution of $\theta=(\theta_1,\ldots,\theta_p)$, based on which Bayesian estimators and credible regions can be constructed. Their frequentist validity together with second order properties can be rigorously justified by our Theorem~\ref{thm:MAIN1}. For example, Theorem~\ref{thm:MAIN1} directly implies the semiparametric efficiency of the posterior median as follows.
\begin{corollary}\label{coro:1}
Consider the semiparametric model and the prior $\Pi$ in Theorem \ref{thm:MAIN1}. Under the same assumptions, the coordinate-wise marginal posterior median $\widehat{\theta}^B_n$ satisfies
\begin{align*}
    \sqrt{n}(\widehat{\theta}^B_n-\theta_0)=\widetilde{\Delta}_n+O_{P_0}(S_n),
\end{align*}
where $\widetilde{\Delta}_n\overset{P_{0}}{\rightsquigarrow}N_p(0, \widetilde{I}_{0}^{-1})$ and $S_n=R_n(n^{-1/2}\log n)+\delta_n$.
\end{corollary}
\noindent The conclusion in Corollary~\ref{coro:1} may also hold for posterior mode, but this would require the convergence of posterior density instead of posterior distribution as in Theorem~\ref{thm:MAIN1}.

We next study the frequentist property of credible regions. For any $\alpha\in(0,1)$, we define the $\alpha$-th marginal posterior quantile $\widehat{q}_{s,\alpha}$ of $\theta_s$ through the following equation
$\Pi(\theta_s\leq\widehat{q}_{s,\alpha}|X_1,\ldots,X_n)=\alpha$. Let $(-\infty,q_{s,\alpha}]$ be a one-sided confidence interval for $\theta_s$ of significance level $\alpha$ based on the $s$th component of the best regular estimator, which is well approximated by $\theta_0+\widetilde{\Delta}_n/\sqrt{n}$. In other words, $q_{s,\alpha}$ is given by $\theta_{0,s}+\widetilde{\Delta}_{n,s}/\sqrt{n}+n^{-1/2}(\widetilde{I}_0^{ss})^{1/2}z_{\alpha}$ so that $P_0(\theta_{0,s}\leq q_{s,\alpha})\to \alpha$ as $n\to\infty$. Here $\widetilde{I}_0^{ss}$ is the $(s,s)$-th component of $\widetilde{I}_0^{-1}$, $\theta_{0,s}$ and $\widetilde \Delta_{n,s}$ are the $s$-th components of $\theta_0$ and $\widetilde\Delta_{n}$, respectively. The following corollary suggests that the credible interval $(-\infty,\widehat{q}_{s,1-\alpha}]$ ($[\widehat{q}_{s,\alpha/2}, \widehat{q}_{s,1-\alpha/2}]$) estimates this one-(two-)sided confidence interval for $\theta_s$ of significance level $(1-\alpha)$ with an error of order $S_n$.

\begin{corollary}\label{coro:2}
Consider the semiparametric model and the prior $\Pi$ in Theorem \ref{thm:MAIN1}. Under the same assumptions, we have
$ \sqrt{n}\,|\widehat{q}_{s,\alpha}-q_{s,\alpha}|=O_{P_0}(S_n)$ for $s=1,\ldots,p$.
\end{corollary}

\begin{remark}
The MCMC samples can also be used to construct an estimator of the asymptotic variance $V^*$ (or the efficient information matrix $\widetilde{I}_{0}$), denoted as $\widehat{V}^\ast$. As shown below, we have $\|\widehat{V}^\ast-V^\ast\|_F=O_{P_0}(S_n)$ and $\|(\widehat{V}^\ast)^{-1}-\widetilde{I}_{0}\|_F=O_{P_0}(S_n)$, where $\|\cdot\|_F$ is the Frobenius norm. The diagonal element $V^\ast_{ss}$ can be estimated by $\widehat{V}^\ast_{ss}=\sqrt{n}(\widehat{q}_{s,1-\alpha/2}-\widehat{q}_{s,\alpha/2})/(2z_{1-\alpha/2})$, where $z_{\alpha}$ is the $\alpha$th quantile of a standard normal distribution. According to the proof of Corollary \ref{coro:2}, we have $\widehat{q}_{s,\alpha}=\theta_{0,s}+n^{-1/2}\widetilde{\Delta}_{n,s}+n^{-1/2}(V^\ast_{ss})^{1/2}z_{\alpha}+n^{-1/2}\,O_{P_0}(S_n)$, which implies
$\widehat{V}^\ast_{ss}-V_{ss}^\ast=O_{P_0}(S_n)$. For the off-diagonal element $V^\ast_{ss'}$ ($s\neq s'$), we can first obtain the $\alpha$th quantile $\widehat{q}_{s,s',\alpha}$ for the marginal posterior distribution of $\vartheta=\theta_s+\theta_{s'}$ and then set $\widehat{V}^\ast_{ss'}=\frac{1}{2}\big\{\sqrt{n}(\widehat{q}_{s,s',1-\alpha/2}
-\widehat{q}_{s,s',\alpha/2})/(2z_{1-\alpha/2})-\widehat{V}^\ast_{ss}-\widehat{V}^\ast_{s's'}\big\}$. Since equation~\eqref{eq:sBvM} implies that
$$
 \sup_A\big|\Pi(\vartheta\in A|X_1,\ldots,X_n)-N\big(\theta_{0,s}+\theta_{0,s'}+n^{-1/2}\widetilde{\Delta}_{n,s}+
 n^{-1/2}\widetilde{\Delta}_{n,s'},n^{-1}\Sigma\big)(A)\big|= O_{P_0}(S_n),
$$
where $\Sigma=V^\ast_{ss}+V^\ast_{s's'}+2V^\ast_{ss'}$, we obtain $\widehat{V}^\ast_{ss'}=V^\ast_{ss'}+O_{P_0}(S_n)$. This proves our previous claim.
\end{remark}

\subsection{Verification of Assumptions~\ref{A.2} and \ref{A.3}}\label{sec:assver}

We verify Assumption~\ref{A.2} in a general class of statistical models
$\mathcal{P}=\{P^{(n)}_{\lambda}: \lambda\in\mathcal F\}$, where the observations $Y^{(n)}=(Y_1,\ldots,Y_n)$ are independent but not necessarily identically distributed. Hence, we have $P^{(n)}_{\lambda}(Y^{(n)})\equiv\prod_{i=1}^n P_{\lambda,i}(Y_i)$ with $P_{\lambda,i}$ the marginal distribution of $Y_i$ under a common parameter $\lambda$ (whose true value is denoted as $\lambda_0$). In the above setup, \cite{Ghosal2007} derived the posterior contraction rate of $\lambda$ as being at least $\xi_n$ (in terms of a semi-metric $d_n^2(\lambda,\lambda')\equiv\frac{1}{n}\sum_{i=1}^n\int(\sqrt{p_{\lambda,i}}-\sqrt{p_{\lambda',i}})^2d\mu_i$ for any pair $(\lambda,\lambda')$ in $\mathcal{F}$) by showing $\Pi\big(d_n(\lambda,\lambda_0)\geq M\xi_n\big|X_1,\ldots,X_n\big)=o_{P_{\lambda_0}^{(n)}}(1)$. In Lemma~\ref{le:pcr2} below, we obtain an exponential convergence rate of $\Pi\big(d_n(\lambda,\lambda_0)\geq M\xi_n \big|X_1,\ldots,X_n\big)$ by keeping track of the remainder term in the proof of Theorem 4 therein.

Lemma~\ref{le:pcr2} is also of independent interest. Denote $V_{2}(P,Q)=\int |\log(dP/dQ)-K(P,Q)|^2dP$ as a discrepancy measure between two probability measures $P$ and $Q$.
\begin{lemma}\label{le:pcr2}
Let $\xi_n$ be a sequence satisfying $\xi_n\to0$ and $n\xi_n^2\to\infty$. If there exists an increasing sequence of sieves $\mathcal{F}_n\subset\mathcal{F}$, such that the following conditions are satisfied:
\begin{enumerate}
  \item[a.] $\Pi(\mathcal{F}\backslash\mathcal{F}_n)\leq \exp(-n\xi_n^2(C+4))$\;\;\mbox{for some $C>0$};
  \item[b.] $\log N(\xi_n,\mathcal{F}_n,d_n)\leq n\xi_n^2$;
  \item[c.] $\Pi(B_n(P_{0}^{(n)},\xi_n))\geq \exp(-Cn\xi_n^2)$,
\end{enumerate}
where $B_n(P_{0}^{(n)},\xi_n)=\Big\{\lambda\in \mathcal{F}: K(P^{(n)}_0, P_{\lambda}^{(n)}) \leq n\xi_n^2, \, V_{2}(P^{(n)}_{0}, P_{\lambda}^{(n)})\leq n\xi_n^2\Big\}$, then for some constant $C_1>0$ and large enough $M$, we have
\begin{align}
    \Pi\big(d_n(\lambda,\lambda_0)\geq M\xi_n\big|X_1,\ldots,X_n\big)=O_{P_{\lambda_0}^{(n)}}(\exp(-C_1n\xi_n^2)).\label{eq:dpcr2}
\end{align}
\end{lemma}

In semiparametric models, the sieve sequence $\mathcal F_n$ typically consists of one parametric part and one nonparametric part. For example, $\mathcal{F}_n=\mathcal{F}_n^{\theta}\oplus\mathcal{F}_n^{\eta}=\{\theta^Tu+\eta(v):
\theta\in\mathcal{F}_n^{\theta}, \ \eta\in\mathcal{F}_n^{\eta}\}$ in the class of GPLM. By viewing $(\theta,\eta)$ as $\lambda$ in the above lemma, we can conclude that the posterior probability of the event $\{\|U^T(\theta-\theta_0)+\eta-\eta_0\|_n \leq M\xi_n\}$ is $1-O_{P_{\lambda_0}^{(n)}}(\exp(-C_1n\xi_n^2))$ if $d_n(\lambda,\lambda_0)$ dominates $\|U^T(\theta-\theta_0)+\eta-\eta_0\|_n$. In partially linear models, we can further show that $\{|\theta-\theta_0|\leq c\xi_n, \|\eta-\eta_0\|_n\leq c\xi_n\}$ for some constant $c>0$ given that the matrix $P_0(U-E[U|V])^{\otimes 2}$ is invertible. Please see Lemma~\ref{LemmaCM} and the arguments after that. In this case, we know that $\rho_n$ and $\epsilon_n$ in Assumption~\ref{A.2} turn out to be $\xi_n$ given in Lemma~\ref{le:pcr2} (and $\delta_n=\exp(-C_1n\xi_n^2)$). As a by-product of Lemma~\ref{le:pcr2}, we show that $\Pi\big(\lambda\not\in \mathcal{F}_n\big|X_1,\ldots,X_n\big)=O_{P_{\lambda_0}^{(n)}}(e^{-C_1n\xi_n^2})$ by following Lemma 1 in \cite{Ghosal2007}. In the end, we remark that Lemma~\ref{le:pcr2} does not apply to generalized partial linear models. Rather, we verify Assumption~\ref{A.2} by directly applying Lemma 2 in \cite{Ghosal2007}; see Lemma~\ref{le:pcr4}.

We next discuss the sufficient condition (A1) for Assumption \ref{A.3}. Note that (A1) depends on the prior through the localization sequence $\{\mathcal{H}_n\}$ in Assumption~\ref{A.2}, to which the posterior distribution allocates most mass. With a small subset $\mathcal{H}_n$, the L.H.S. of (\ref{eq:A1}) converges to zero at a faster rate. Hence, we want to make $\mathcal{H}_n$ as small as possible while keeping $\Pi(\mathcal{H}_n|X_1,\ldots,X_n)$ close to one. Motivated by this, we set
\begin{align}\label{eq:hn}
    \mathcal{H}_n=\{\eta:\|\eta-\eta_0\|_n\leq M\rho_n\}\cap\mathcal{F}_n^{\eta},
\end{align}
where $\{\mathcal{F}_n^{\eta}\}$ is the sieve sequence constructed in Lemma \ref{le:pcr2}. By Assumption \ref{A.2} and condition (a) in Lemma \ref{le:pcr2}, we obtain that $\Pi(\mathcal{H}_n|X_1,\ldots,X_n)=1-O_{P_0}(\delta_n)$ with $\delta_n=e^{-n\rho_n^2}$. Then we can bound the L.H.S. of (\ref{eq:A1}) from above by calculating the continuity modulus or applying the maximal inequalities in \cite{Van1996}; see Lemma~\ref{Lemma:MI}. Please see Section \ref{se:egs} for the verification of (A1) in concrete examples.

Now we are ready to state our lemma for Assumption \ref{A.3}.
\begin{lemma}\label{thm:2}
If (A1) and (A2) hold,
then we have the following $R_n=G_n+\widetilde{G}_n$ in Assumption \ref{A.3}.
\end{lemma}

\section{Semiparametric Prior}\label{se:sop}
In this section, we consider two classes of priors, differing in whether $\theta$ and $\eta$ are dependent, and then specify the corresponding form of $\widetilde G_n(\cdot)$ in the prior stability condition (A2) for them. In general, in applying the semiparametric BvM theorem we find that the dependent prior has advantages in requiring less stringent conditions and being adaptive to the unknown smoothness of the nonparametric function.
Throughout this section, we impose a smoothness condition on the least favorable curve:

(A3) There exists a function $h^\ast\in L_2(P_0)$, referred to as least favorable direction, such that
$$
\Delta\eta(\theta)=(\theta-\theta_0)^Th^\ast+O(|\theta-\theta_0|^2)\;\;\mbox{as}\;\theta\to\theta_0.
$$
Note that (A3) is commonly assumed in the literature. For example, it holds for the class of GPLM under mild conditions; see Lemma \ref{le:glfc}.

\subsection{Independent Prior}\label{se:sop1}
Consider a pair of independent priors:
\begin{align*}
  \text{(PI)}&&  \theta\sim\Pi_{\Theta},\quad \eta\sim\Pi_{H}.&&
\end{align*}
This is a common choice in the semiparametric Bayesian literature with various forms of $\Pi_{\mathcal{H}}$. For example, \citet{Kim2006} considered a class of neutral-to-the-right process priors for the cumulative hazard function in the Cox proportional hazard model, while \citet{Castillo2012} considered a class of Gaussian process priors \citep{Rasmussen2006} for the same model. Another example is a Riemann-Liouville type prior considered by \citet{bickel2012} in the partially linear models.

We next specify the form of $\widetilde G_n(\cdot)$ under the above independent prior. For technical reasons, we need to introduce a sequence of approximations to the least favorable direction $h^\ast$, denoted as $\{h_n\}$. Let $\Pi_{H,\cdot-g}$ represent the distribution of $W-g$ for $W\sim\Pi_H$ and a function $g$, and define $f_n=d\Pi_{H,\cdot-(\theta_n-\theta_0)^Th_n}/d\Pi_H$ as the Radon-Nykodym derivative. For any set $A\subset \mathcal{H}$ and element $f\in \mathcal{H}$, let $A-f$ denote the set $\{g-f:\, g\in A\}$. For $\epsilon_n$ and $\delta_n$ specified in Assumption~\ref{A.2}, we assume that

(A4) There exists a nondecreasing function $\bar{G}_n:\bbR\to\bbR$, such that for any $\theta_n=\theta_0+O_{P_0}(\epsilon_n)$,
 \begin{align*}
   \sup_{\eta\in\mathcal{H}_n}\Big| l_n\big(\theta_0, \eta-\Delta\eta(\theta_n)+(\theta_n-\theta_0)^Th_n\big)-l_n(\theta_0,\eta)\Big|
   =O_{P_0}(\bar G_n(|\theta-\theta_0|)).
\end{align*}

(A5) For any $\theta_n$ satisfying $|\theta_n-\theta_0|\leq \epsilon_n$, we have
 \begin{align*}
   \Pi_H\big(\eta\in \mathcal{H}_n-(\theta_n-\theta_0)^Th_n\big| X_1,\ldots,X_n\big) = 1-O_{P_0}(\delta_n).
\end{align*}

(A6) For any $\theta_n=\theta_0+O_{P_0}(\epsilon_n)$, $|\log f_n(\eta)|=O_{P_0}[\bar G_n(|\theta_n-\theta_0|)]$ holds with $\eta\sim \Pi_H$.

(A4) characterizes the robustness of $l_n(\cdot)$ against a small perturbation in $\eta$. In fact, by Condition (A3), we have $\Delta\eta(\theta_n)-(\theta_n-\theta_0)^Th_n=(\theta_n-\theta)^T(h^\ast-h_n)+O(|\theta_n-\theta_0|^2)$. Hence, Condition (A4) is expected to hold if $h_n$ is sufficiently close to $h^\ast$. Similar to (A4), (A5) characterizes the concentration stability of the localization sequence $\{\mathcal{H}_n\}$ against a small perturbation in $\eta$. This stability can be easily obtained by slightly enlarging the localization sequence via $\mathcal{H}_n\mapsto
\bigcup_{|\theta-\theta_0|\leq \epsilon_n}\big\{\mathcal{H}_n-(\theta-\theta_0)^Th_n\big\}$. For simplicity, we tacitly assume that this enlargement is always made for $\mathcal{H}_n$. As we will clarify in the proof of Theorem~\ref{thm:4a,3a}, this enlargement only increases the covering entropy of $\mathcal{H}_n$ by a negligible amount proportional to $\log(\epsilon_n^{-1})$, which will not affect our results. (A6) characterizes the robustness of the marginal prior $\Pi_H$ against a small perturbation.
The reason for introducing the approximation sequence $\{h_n\}$ is that the Radon-Nykodym derivative $|\log f_n(\eta)|$ in (A6) might have peculiar behavior at $h_n=h^\ast$. As an example, we consider the partially linear model in Section \ref{se:egs} where a Gaussian process (GP) prior $\Pi_{\mathcal{H}}$ is assigned. If we set $h_n$ as $h^\ast$, then we have to require $h^\ast\in\mathbb{H}$\footnote{$\mathbb{H}$ is the reproducing kernel Hilbert space (RKHS) associated with the assigned GP} such that $|\log f_n(\eta)|$ converges to zero. This requirement is very strict since $\mathbb H$ is often a very small subset of $\mathcal H$. Fortunately, we can always find an approximation sequence $\{h_n\}\subset\mathbb{H}$ under which condition (A6) is satisfied. Note that a similar condition to (A6) is also required for the first order semiparametric BvM theorem; see \cite{Castillo2012}.

To verify the stability condition (A2), we can decompose the semiparametric bias $\Delta\eta(\theta_n)$ into two components: $\Delta\eta(\theta_n)-(\theta_n-\theta_0)^Th_n$ and $(\theta_n-\theta_0)^Th_n$. The former can be dealt with (A4) through likelihood and the latter by (A5) and (A6) through the localization sequence and the prior. This is summarized in the following lemma.

\begin{lemma}\label{eq:pinp}
Suppose that Conditions (A4), (A5) and (A6) hold. Then the pair of independent priors (PI) satisfies (A2) with $\widetilde{G}_n(t)=\bar{G}_n(t)+\delta_n$.
\end{lemma}

\subsection{Dependent Prior}\label{se:depp}

In this section, we construct a class of dependent priors $(\Pi_\Theta, \Pi_H^\theta)$. Dependent priors facilitate the development of adaptive Bayesian procedures that do not require knowledge of the smoothness of $\eta$ in specifying $\Pi_H^{\theta}$.  This adaptiveness is achieved by correcting the $\theta$-dependent bias $\Delta\eta(\theta)$ in the prior construction. We remark that adaptiveness cannot be achieved by the independent priors (see Section \ref{se:ipac}), and this finding is consistent with the negative observations in \cite{Rivoirard2012} for linear functionals of densities.

Let $\widehat{h}_n$ be an estimator of the least favorable direction $h^\ast$ that satisfies (A4) -- (A5) with $h_n=\widehat{h}_n$.
Again, by (A3) we have $\Delta\eta(\theta_n)-(\theta_n-\theta_0)^T\widehat{h}_n=(\theta_n-\theta_0)^T(h^\ast-\widehat{h}_n)+O(|\theta_n-\theta_0|^2)$. Consequently, (A4) is implied by the following condition with $\bar G_n(t)=n\rho_n\kappa_n t+n\rho_n t^2$:

\vspace{0.3cm}
(A7) The estimator $\widehat{h}_n$ of $h^\ast$ satisfies $\|\widehat{h}_n-h^\ast\|_n=O_{P_0}(\kappa_n),\ \kappa_n\to 0$.
\vspace{0.3cm}
Please see concrete examples in Section \ref{se:egs} for more discussion on Condition (A7). Let $\Pi_{\Theta}$ be a marginal prior for $\theta$ that satisfies the condition in Theorem~\ref{thm:MAIN1} and $\Pi_{\mathcal{H}}$ a prior for $\eta$. Consider the following joint prior distribution for $(\theta,\eta)$,
\begin{align*}
  \text{(PD)}&&  \theta\sim\Pi_{\Theta},\quad \eta|\theta\sim W+\theta^T\widehat{h}_n \text{ with }W\sim\Pi_{\mathcal{H}}.&&
\end{align*}
The conditional prior distribution $\Pi_{\mathcal{H}}^{\theta}$ of $\eta$ given $\theta$ is obtained by shifting the center of $\Pi_{\mathcal{H}}$ by a $\theta$-dependent amount, i.e., $\theta^T\widehat{h}_n$. By introducing this dependent structure, we can compensate for the semiparametric bias without imposing Condition (A6). In the end, we remark that the randomness of $\widehat{h}_n$ only enters equation~\eqref{eq:ilanb} in Assumption~\ref{A.3} through the remainder term, and thus can be decoupled from the randomness in the leading terms of equation~\eqref{eq:ilanb}. Hence, the proof of Theorem~\ref{thm:MAIN1} still goes through even though (PD) is data-dependent. This is an appealing feature of the proposed prior; our theory shows that we do not need to split the sample and apply a two stage approach to obtain a valid characterization of uncertainty. This is backed up by our simulations.

\begin{lemma}\label{le:dprior}
If conditions (A4) -- (A5) are met with $h_n=\widehat{h}_n$, then the dependent prior (PD) satisfies (A2) with $\widetilde{G}_n=\bar{G}_n+\delta_n$.
\end{lemma}

\subsection{Second-order BvM Theorem under Independent/Dependent Prior}

We summarize the discussions on independent prior (PI) and dependent prior (PD) in the following theorem, which is a straightforward application of Theorem \ref{thm:MAIN1}.

\begin{theorem}\label{thm:MAIN2}
Suppose $X_1,\ldots,X_n$ are i.i.d.\! observations sampled from $P_0=P_{\theta_0,\eta_0}$. Suppose that Assumption \ref{A.2}, Conditions (A1) and (A3) hold and the prior for $\theta$ is dense at $\theta_0$.  We further assume Conditions (A4) -- (A6) for the independent prior (PI) and Conditions (A4) -- (A5) for the dependent prior (PD). Then
the marginal posterior for $\theta$ has the following expansion in total variation as $n\to\infty$,
\begin{align*}
    \sup_A\big|\Pi(\theta\in A|X_1,\ldots,X_n)-&N_p\big(\theta_0+n^{-1/2}\widetilde{\Delta}_n,(n\widetilde{I}_{0})^{-1}\big)(A)\big|\\
    &= O_{P_0}[G_n(n^{-1/2}\log n)+\bar{G}_n(n^{-1/2}\log n)+\delta_n].
\end{align*}
\end{theorem}

\section{Examples}\label{se:egs}

In this section, we construct specific priors for three semiparametric models: partially linear model (PLM), GPLM and the Cox regression model. In PLM, we consider two scenarios: (i). the smoothness of the nonparametric part $\eta$ is known; (ii). the smoothness is unknown and an adaptive marginal prior is assigned to $\eta$. The non-adaptive and adaptive results obtained in PLM can be easily generalized to GPLM. We assign GP priors for the first two models and a Riemann-Liouville type prior for the last model.
\subsection{Partially Linear Model}\label{se:plm}


\subsubsection{Non-adaptive Bayesian Procedure}

We start with a pair of independent priors. In principle, the marginal prior for the parametric part $\theta$ can be any continuous distribution with full support over $\Theta$. For computational convenience such as conjugacy, we specify $\Pi_{\Theta}$ as a multivariate normal distribution $N(0,I_p/\phi_0)$ with $\phi_0$ the precision parameter. For example, one can choose $\phi_0=0.01$ to induce a vague prior for normalized predictors. For the nuisance part, we choose $\Pi_{\mathcal{H}}$ as a stationary Gaussian process (GP) prior $GP(m,K^a)$ indexed by an inverse bandwidth parameter $a$ \citep{Van2009}. Here, the notation $GP(m,K)$ denotes a Gaussian process with mean function $m:\bbR^d\rightarrow \bbR$ and covariance function $K: \bbR^d\times\bbR^d\rightarrow \bbR$. The scaled covariance function $K^a$ is defined as $K^a(x,y)=K_0(ax,\,ay)$, where $K_0$ is a base covariance function\footnote{For the covariance function $K^a$, we use $\mathbb{H}^a$ and $\|\cdot\|_a$ to denote the associated RKHS and RKHS norm, respectively. The unit ball in the RKHS $\mathbb{H}^a$ is denoted by $\mathbb{H}^a_1$.}. Through this section, we focus on the squared exponential covariance function $K_0(x,y)=\exp(-|x-y|^2)$. We next discuss the choice for the inverse bandwidth parameter $a$ given the knowledge of the smoothness of the nuisance function, denoted as $\alpha$.
Given $n$ independent observations, the minimax rate of estimating a $d$-variate $\alpha$-smooth function is known to be $n^{-\alpha/(2\alpha+d)}$ \citep{Stone1982}.
\citet{Van2009} showed that with $a_n=n^{1/(2\alpha+d)}$ the Gaussian process prior
$GP(0,K^{a_n})$ leads to the minimax rate up to a $\log n$ factor. Hence, we set $a_n=n^{1/(2\alpha+d)}$ in this subsection.

We next focus on the dependent prior (PD). The least favorable direction $h^\ast(\cdot)$ in this model is essentially $-E[U|V=\cdot]$, which can be directly estimated based on the design points $\{(U_i,V_i)\}$, e.g. by kernel method. Denote this estimator as $\widehat{h}_n$. Since shifting the center of GP is equivalent to translating its mean function, we can write the dependent prior (PD) as
\begin{align*}
    \theta\sim&\ \Pi_{\Theta}\ \text{ and }\ \eta\, |\, \theta \sim  GP(\theta^T\widehat{h}_n,K^{a_n}).
\end{align*}
By writing (PD) in this form, we can discuss its relation with independent priors. If we reparameterize the nuisance parameter by $\xi=\eta-\theta^T\widehat{h}_n$, then $\xi|\theta\sim GP(0,K^{a_n})$ and the partially linear model becomes
\begin{align}\label{eq:newplm}
    Y=\theta^T\big[U+\widehat{h}_n(V)\big]+\xi(V)+w,
\end{align}
with true values $\theta_0$ and $\xi_0:=\eta_0-\theta_0^T\widehat{h}_n$. If we treat $U+\widehat{h}_n(V)$ as a new covariate $\widetilde{U}$, then the least favorable direction of the new model becomes
\begin{align*}
    \widetilde{h}=E[\widetilde{U}|V]=\widehat{h}_n-h^\ast=O_{P_0}(\kappa_n)\overset{P_{0}}{\longrightarrow} 0,
\end{align*}
where $\kappa_n$ is defined in (A7). This suggests that the semiparametric bias $\Delta\xi(\theta)$ of the new model is negligible.

\begin{theorem}\label{thm:4a,3a}
Let $X_i=(U_i,V_i,Y_i)\in\bbR^p\times\bbR^d\times\bbR$, $i=1,\ldots,n$, be $n$ observations from the partially linear model \eqref{eq:plmf}. Assume the following conditions:
\begin{enumerate}
  \item $\eta_0$ is H\"{o}lder $\alpha$-smooth, where $\alpha>d/2$.
   \item The information matrix $\widetilde{I}_0=P_0(U-E[U|V])(U-E[U|V])^T$ is invertible.
  \item For the independent prior, the conditional expectation $E[U|V=v]$ as a function of $v$ is at least $\alpha$-smooth; for the dependent prior, (A7) holds with $\kappa_n=O_{P_0}(\rho_n)$, where $\rho_n=n^{-\alpha/(2\alpha+d)}(\log n)^{1+d}$.
\end{enumerate}
Then with the choice of $a_n=n^{1/(2\alpha+d)}$, we have the following second order BvM result:
\begin{equation}\label{eq:plmBvM}
    \sup_A\big|\Pi(\theta\in A|X^{(n)})-N_p\big(\theta_0+n^{-1/2}\widetilde{\Delta}_n,(n\widetilde{I}_0^{-1})\big)(A)\big|= O_{P_0}(\sqrt{n}\rho_n^2\log n)=O_{P_0}\big(n^{-\frac{\alpha-d/2}{2\alpha+d}}(\log n)^{2d+3}\big),
\end{equation}
where $\widetilde{\Delta}_n=n^{-1/2}\sum_{i=1}^n\widetilde{I}_0^{-1}w_i(U_i-E[U|V_i])
\overset{P_0}{\rightsquigarrow} N_p(0,\widetilde{I}_0^{-1})$.
\end{theorem}

If the smoothness of the GP does not match with the smoothness of the regression function, i.e. $a_n=n^{1/(2\alpha'+d)}$ with $\alpha'\neq\alpha$, then the convergence rate of the nuisance parameter provided by Theorem~\ref{thm:4a,3a} becomes suboptimal: $\rho_n=n^{-\alpha'/(2\alpha' +d)}$ when $\alpha'<\alpha$ and $\rho_n=n^{-\alpha/(2\alpha'+d)}$ when $\alpha'>\alpha$. Therefore, it is crucial to choose a proper nonparametric prior for obtaining a better frequentist accuracy of the semiparametric Bayesian procedure. In the end, we remark that the remainder term in the above fully Bayesian framework matches with that derived in \cite{cheng2008b, cheng2009} with the nonparametric part profiled out. Note that \cite{cheng2008b, cheng2008, cheng2009} is not a fully Bayesian framework, and did not cover the adaptive case. However, the adaptiveness can be easily incorporated into the construction of our nonparametric Bayesian prior as will be seen in the next section.

\subsubsection{Adaptive Bayesian Procedure}
In the adaptive case, we still specify a GP prior for $\Pi_H$. To allow adaptation to the unknown smoothness $\alpha$, we follow \citet{Van2009} by putting a prior on the inverse bandwidth $A$.
\citet{Van2009} showed that the hierarchical prior
\begin{align}\label{eq:GPprior}
W^A|A\sim GP(0,K^A),\quad A^d\sim Ga(a_0,b_0)
\end{align}
with $Ga(a_0,b_0)$ the Gamma distribution whose pdf $p(t)\propto t^{a_0-1}e^{-b_0t}$ leads to the minimax rate $n^{-\alpha/(2\alpha+d)}$ up to $\log n$ factors, adaptively over all  smoothness $\alpha>0$. Since the choice of hyper-parameters has a diminishing impact on the posterior distribution as the sample size $n$ grows, we simply choose $a_0=1$ and $b_0=1$.

Under the above choice for $\Pi_H$, Condition (A6) becomes overly stringent, suggesting the incompetence of the independent prior in the adaptive scenario. Please see Section \ref{se:ipac} in the Appendix for further explanation. Fortunately, the dependent prior avoids (A6) by incorporating the bias correction as follows:
\begin{equation}\label{eq:saprior}
\begin{aligned}
    \theta\sim&\ \Pi_{\Theta}, \quad A^d\sim Ga(a_0,b_0),\\
    &\eta\,|\,\theta,A \sim  GP(\theta\widehat{h}_n,K^A).
\end{aligned}
\end{equation}

The proof of Theorem \ref{thm:4} is omitted due to its similarity to those of Theorems \ref{thm:4a,3a} and \ref{thm:3a'}.
\begin{theorem}\label{thm:4}
Let $X_i=(U_i,V_i,Y_i)$, $i=1,\ldots,n$, be a sample from the partially linear model. Suppose that $\widehat{h}_n$ is an estimator of the least favorable direction  $h^\ast(\cdot)=-E[U|V=\cdot]$, and conditions (i)-(iii) in Theorem \ref{thm:4a,3a} hold for the dependent prior.
Then under the prior \eqref{eq:saprior}, the following second order BvM result holds:
\begin{equation}\label{eq:plmBvM}
    \sup_A\big|\Pi(\theta\in A|X^{(n)})-N_p\big(\theta_0+n^{-1/2}\widetilde{\Delta}_n,(n\widetilde{I}_0^{-1})\big)(A)\big|=O_{P_0}\big(n^{-\frac{\alpha-d/2}{2\alpha+d}}(\log n)^{2d+3}\big).
\end{equation}
\end{theorem}
Note that the remainder term in Theorem~\ref{thm:4} is exactly the same as that in Theorem~\ref{thm:4a,3a}. However, we cannot claim that the adaptive procedure does not lead to any loss of the second order Bayesian efficiency because the remainder term is not proven to be sharp.

\subsubsection{Simulation Results}\label{se:sim}
In this section, we conduct a simulation study for comparing the dependent and independent prior in the adaptive scenario. In each setting, we generated $100$ datasets from the following four models:
\begin{description}
  \item[M1] $Y_i=0.5 U_i + \exp(V_i) + N(0,0.5^2)$, with $V_i\overset{iid}\sim N(0,1)$ and $U_i|V_i\sim N(0.5 |V_i|^3,1)$;
  \item[M2] $Y_i=0.5 U_i + \exp(V_i) + N(0,0.5^2)$, with $V_i\overset{iid}\sim N(0,1)$ and $U_i|V_i\sim N(0.5 V_i^3,1)$;
  \item[M3] $Y_i=0.5 U_i + \exp(|V_i|) + N(0,0.5^2)$, with $V_i\overset{iid}\sim N(0,1)$ and $U_i|V_i\sim N(0.5 |V_i|^3,1)$;
  \item[M4] $Y_i=0.5 U_i + \exp(|V_i|) + N(0,0.5^2)$, with $V_i\overset{iid}\sim N(0,1)$ and $U_i|V_i\sim N(0.5 V_i^3,1)$.
\end{description}
In M1, the least favorable direction $h^\ast(v)=0.5|v|^3$ is twice differentiable but not thrice differentiable at $v=0$. In contrast, the least favorable direction $h^\ast(v)=0.5v^3$ in M2 is infinitely differentiable. M3 and M4 are counterparts of M1 and M2 respectively with non-differentiable nuisance parts at $v=0$. As for assigned priors, we consider three different setups: P1. the independent prior with $\Pi_H$ specified by \eqref{eq:GPprior}; P2. the dependent prior \eqref{eq:saprior} with an estimator $\widehat{h}_n(v)$ produced by the Nadaraya-Watson kernel regression method\footnote{We apply the Gaussian kernel with an optimal bandwidth \citep[p.31]{adrian1997}}; P3. the dependent prior \eqref{eq:saprior} with $\widehat{h}_n(v)=-E(U|V=v)$. In each, we chose a vague prior $N(0,10^2)$ as $\Pi_{\Theta}$, and hyper-parameters $a_0=b_0=1$. For each replicate, we ran MCMC for $10,000$ iterations and discarded the first $5,000$ as the burn-in.

The results for M1 and M2 are displayed in Table \ref{tab:1}. We varied the sample size $n$ from $50$ to $400$ and applied the three priors P1, P2 and P3. We record the root mean squared error (RMSE) for $\theta$ (under the Euclidean norm) and $\eta$ (under the empirical norm), respectively, across $100$ replicates. The average estimated standard error based on MCMC (SE) and the empirical coverage of nominal $95\%$ credible intervals based on MCMC (CR95) are also reported. From Table \ref{tab:1}, we can see that the estimation accuracy of $\theta$ (in terms of RMSE) improves under the dependent priors P2 and P3 as $n$ grows. However, the RMSE for $\theta$ under the independent prior P1 only significantly decreases as $n$ goes from $50$ to $100$, and remains around $0.1$ thereafter. On the other hand, the estimated standard errors produced by P1 -- P3 are all very close. The CR95 results further illustrate the significant under-coverage of the credible intervals produced by P1. All of these empirical observations justify the existence of semiparametric bias (discussed after Lemma \ref{thm:2}), and illustrate the necessity to compensate this bias by using the dependent priors. Moreover, we observe that the RMSE for $\theta$ intimately depends on that for $\eta$: a large RMSE for $\eta$ usually leads to a large RMSE for $\theta$, which is consistent with our theory. For example, Bayesian estimation accuracy of $\theta$ is higher in M1 than in M2. Another observation from Table \ref{tab:1} is that as $n$ increases, the difference in estimation accuracy between P2 and P3 becomes negligible. This might be attributed to the increasing accuracy of the estimation of $\widehat{h}_n$.

\begin{table*}[htp!]
\footnotesize
  \centering
\begin{tabular}{c|c|c|C{1.4cm}C{1.4cm}C{1.4cm}C{1.4cm}}
\hhline{=======}
   & model & method & RMSE($\theta$) & SE & RMSE($\eta$) & CR95 \\
\hline
         \multirow{6}{*}{$n=50$} & \multirow{3}{*}{M1} & P1 & 0.115 & 0.078 & 0.308 & 0.92 \\
       & & P2 & 0.085 & 0.082 & 0.274 & 0.96 \\
       & & P3 & 0.083 & 0.083 & 0.270 & 0.95 \\
               \hhline{~------}
      & \multirow{3}{*}{M2} & P1 & 0.104 & 0.080 & 0.298 & 0.84 \\
      & & P2 & 0.084 & 0.085 & 0.267 & 0.96 \\
      & & P3 & 0.082 & 0.085 & 0.268 & 0.96 \\
      \hline
      \multirow{6}{*}{$n=100$} & \multirow{3}{*}{M1} & P1 & 0.103 & 0.052 & 0.225 & 0.83 \\
       & & P2 & 0.056 & 0.056 & 0.202 & 0.95 \\
       & & P3 & 0.053 & 0.056 & 0.204 & 0.96 \\
               \hhline{~------}
      & \multirow{3}{*}{M2} & P1 & 0.096 & 0.051 & 0.235 & 0.85 \\
      & & P2 & 0.055 & 0.054 & 0.209 & 0.94 \\
      & & P3 & 0.051 & 0.055 & 0.206 & 0.97 \\
      \hline
      \multirow{6}{*}{$n=200$} & \multirow{3}{*}{M1} & P1 & 0.106 & 0.038 & 0.230 & 0.62 \\
       & & P2 & 0.042 & 0.038 & 0.197 & 0.93 \\
       & & P3 & 0.036 & 0.038 & 0.187 & 0.97 \\
               \hhline{~------}
      & \multirow{3}{*}{M2} & P1 & 0.094 & 0.036 & 0.209 & 0.72 \\
      & & P2 & 0.038 & 0.038 & 0.180 & 0.95 \\
      & & P3 & 0.038 & 0.038 & 0.183 & 0.98 \\
      \hline
      \multirow{6}{*}{$n=400$} & \multirow{3}{*}{M1} & P1 & 0.115 & 0.035 & 0.289 & 0.38 \\
       & & P2 & 0.030 & 0.028 & 0.187 & 0.93 \\
       & & P3 & 0.025 & 0.028 & 0.187 & 0.98 \\
               \hhline{~------}
      & \multirow{3}{*}{M2} & P1 & 0.107 & 0.033 & 0.268 & 0.45 \\
      & & P2 & 0.030 & 0.027 & 0.178 & 0.92 \\
      & & P3 & 0.027 & 0.026 & 0.179 & 0.98 \\
             \hhline{=======}
\end{tabular}
  \caption{Simulation results for the partially linear model with a smooth nuisance function based on $100$ replicates. }\label{tab:1}
\end{table*}

Table \ref{tab:2} provides the results for M3 and M4, where the nuisance function is non-differentiable. As expected, the overall RMSE in Table \ref{tab:2} is worse than that in Table \ref{tab:1}. However, similar overall trends as those in Table \ref{tab:2} are observed. For example, the estimation accuracies of P1 are generally worse than those of P2 and P3, and the semiparametric bias in P1 is more salient under M3 and M4 than under M1 and M2. In addition, the RMSE for $\theta$ produced by P1 under a non-smooth least favorable direction $h^\ast$ is significantly worse than the RMSE under a smooth $h^\ast$. This is consistent with condition (A7), because the semiparametric bias under the independent prior (PI) depends on the smoothness of $h^\ast$.

\begin{table*}[htp!]
\footnotesize
  \centering
\begin{tabular}{c|c|c|C{1.4cm}C{1.4cm}C{1.4cm}C{1.4cm}}
\hhline{=======}
   & model & method & RMSE($\theta$) & SE & RMSE($\eta$) & CR95 \\
\hline
         \multirow{6}{*}{$n=50$} & \multirow{3}{*}{M3} & P1 & 0.243 & 0.088 & 0.499 & 0.74 \\
       & & P2 & 0.090 & 0.085 & 0.279 & 0.94 \\
       & & P3 & 0.084 & 0.087 & 0.280 & 0.97 \\
               \hhline{~------}
      & \multirow{3}{*}{M4} & P1 & 0.194 & 0.089 & 0.408 & 0.80 \\
      & & P2 & 0.084 & 0.087 & 0.270 & 0.97 \\
      & & P3 & 0.084 & 0.087 & 0.265 & 0.97 \\
      \hline
      \multirow{6}{*}{$n=100$} & \multirow{3}{*}{M3} & P1 & 0.217 & 0.064 & 0.441 & 0.67 \\
       & & P2 & 0.061 & 0.056 & 0.233 & 0.93 \\
       & & P3 & 0.057 & 0.056 & 0.231 & 0.93 \\
               \hhline{~------}
      & \multirow{3}{*}{M4} & P1 & 0.122 & 0.052 & 0.309 & 0.84 \\
      & & P2 & 0.059 & 0.055 & 0.221 & 0.96 \\
      & & P3 & 0.058 & 0.055 & 0.219 & 0.95 \\
      \hline
      \multirow{6}{*}{$n=200$} & \multirow{3}{*}{M3} & P1 & 0.189 & 0.036 & 0.410 & 0.53 \\
       & & P2 & 0.042 & 0.039 & 0.215 & 0.94 \\
       & & P3 & 0.042 & 0.039 & 0.212 & 0.97 \\
               \hhline{~------}
      & \multirow{3}{*}{M4} & P1 & 0.106 & 0.042 & 0.271 & 0.77 \\
      & & P2 & 0.041 & 0.038 & 0.204 & 0.98 \\
      & & P3 & 0.040 & 0.038 & 0.203 & 0.97 \\
      \hline
      \multirow{6}{*}{$n=400$} & \multirow{3}{*}{M3} & P1 & 0.194 & 0.041 & 0.429 & 0.21 \\
       & & P2 & 0.035 & 0.029 & 0.207 & 0.95 \\
       & & P3 & 0.031 & 0.028 & 0.205 & 0.95 \\
               \hhline{~------}
      & \multirow{3}{*}{M4} & P1 & 0.115 & 0.033 & 0.282 & 0.65 \\
      & & P2 & 0.033 & 0.028 & 0.193 & 0.94 \\
      & & P3 & 0.030 & 0.028 & 0.193 & 0.96 \\
             \hhline{=======}
\end{tabular}
  \caption{Simulation results for the partially linear model with a non-smooth nuisance function based on $100$ replicates. }\label{tab:2}
\end{table*}

\subsection{Generalized Partially Linear Model (GPLM)}
The semiparametric BvM results for GPLM are similar to those for PLM. Hence, we only focus on the more challenging adaptive scenario in this section. In particular, we consider the same dependent prior as in Section \ref{se:plm}, i.e., GP with a random inverse bandwidth parameter. Define
\begin{align*}
    f(\xi)=\frac{dF(\xi)}{d\xi},\quad l(\xi)=\frac{f(\xi)}{V(F(\xi))},\quad \xi\in\bbR,
\end{align*}
$f_0=f(g_0)$ and $l_0=l(g_0)$.

%
%

\begin{theorem}\label{thm:5}
Let $X_i=(U_i,V_i,Y_i)$, $i=1,\ldots,n$, be a sample from GPLM satisfying Assumption~\ref{A.1} in Section~\ref{SectionA.1}. Suppose Condition (i) -- (iii) for the dependent prior in Theorem~\ref{thm:4a,3a} hold. Moreover, assume
\begin{enumerate}
  \item[(ii').]  The information matrix $\widetilde{I}_0=E_0\big[l_0(T)f_0(T)(U+h^\ast(V))(U+h^\ast(V))^T\big]$ and the identification matrix $P_0(U-E[U|V])(U-E[U|V])^T$ are invertible.
\end{enumerate}
Then under the dependent prior~\eqref{eq:saprior}, the following second order BvM result holds:
\begin{equation*}
    \sup_A\big|\Pi(\theta\in A|X^{(n)})-N_p\big(\theta_0+n^{-1/2}\widetilde{\Delta}_n,(n\widetilde{I}_0^{-1}\big)\big)(A)\big|= O_{P_0}\big(n^{-\frac{\alpha-d/2}{2\alpha+d}}(\log n)^{2d+3}\big),
\end{equation*}
where
$ \widetilde{\Delta}_n=n^{-1/2}\sum_{i=1}^n\widetilde{I}_0^{-1}W_il_0(T_i)(U_i+h^\ast(V_i))
\overset{P_0}{\rightsquigarrow} N_p(0,\widetilde{I}_0^{-1})$.
\end{theorem}

\subsection{Cox Proportional Hazard Model}

In this section, we revisit the Cox proportional hazard model with current status data in Section~\ref{sec:cox}. Recall that we use notation $\theta$ and $\Lambda$ to denote the parametric part and nuisance part in the model, respectively, and that the least favorable direction $h^\ast$ is given by
\begin{align}\label{eq:lfd2}
    h^\ast(c)=-\Lambda_0(c)\frac{E[ZQ^2_{\theta_0,\Lambda_0}(X)|C=c]}{E[Q^2_{\theta_0,\Lambda_0}(X)|C=c]}.
\end{align}
Assume that the true baseline hazard function $\lambda_0$ is Lipschitz continuous and uniformly bounded away from zero. Then, by reparametrizing $\log\lambda$ as the nuisance function $\eta$, we assign the following Riemann-Liouville type prior $\Pi_{\eta}$ \citep[Section 4.2]{Van2008b}
\begin{align}\label{eq:RLprior}
    r(t)=\int_{0}^t(t-u)^{1/2}dW_u+\sum_{k=0}^2Z_kt^k,\quad 0\leq t\leq\tau,
\end{align}
where $Z_k\overset{iid}{\sim}N(0,1)$, $k=0,1,2$. The prior $\Pi_{\Theta}$ can be chosen as any distribution with positive pdf everywhere over $\bbR^p$.

\begin{theorem}\label{thm:6}
Let $X_i=(C_i,\delta_i,Z_i)$, $i=1,\cdots,n$, be a sample from the Cox model with current status data. Assume that $h^\ast$ given by \eqref{eq:lfd2} is Lipschitz continuous and $\Lambda_0$ satisfies $\Lambda_0(\tau)\leq M$ for some constant $M$. Then under the independent prior $\Pi_{\Theta}\times\Pi_{\eta}$, the following second order BvM result holds:
\begin{equation*}
    \sup_A\big|\Pi(\theta\in A|X^{(n)})-N_p\big(\theta_0+n^{-1/2}\widetilde{\Delta}_n,(n\widetilde{I}_{\theta_0,\Lambda_0}^{-1}\big)\big)(A)\big|= O_{P_0}\big(\sqrt{n}\rho_n^2\big),
\end{equation*}
where $\rho_n=n^{-1/3}$,
$\widetilde{\Delta}_n=n^{-1/2}\sum_{i=1}^n\widetilde{I}_{\theta_0,\Lambda_0}^{-1}
    \big(Z_i\Lambda_0(C_i)+h^\ast(C_i)\big)Q_{\theta_0,\Lambda_0}(X_i)
\overset{P_0}{\rightsquigarrow} N_p(0,\widetilde{I}_{\theta_0,\Lambda_0}^{-1})$ and the information matrix $\widetilde{I}_{\theta_0,\Lambda_0}=E_0\big[\big(Z\Lambda_0(C)+h^\ast(C)\big)\big(Z\Lambda_0(C)+h^\ast(C)\big)^TQ^2_{\theta_0,\Lambda_0}(X)\big]$.
\end{theorem}

%
%

\newpage

\vskip 1em \centerline{\Large \bf APPENDIX} \vskip 1em
\setcounter{subsection}{0}
\renewcommand{\thesubsection}{A.\arabic{subsection}}
\setcounter{equation}{0}
\renewcommand{\theequation}{A.\arabic{equation}}

\setcounter{theorem}{0}
\renewcommand{\thetheorem}{A.\arabic{theorem}}
\setcounter{Lemma}{0}
\renewcommand{\theLemma}{A.\arabic{Lemma}}

\subsection{Independent prior for the adaptive procedure}\label{se:ipac}

In this section, we explain why independent priors are not suitable for {\em adaptive} Bayesian procedures. In short, we need to impose a very stringent condition on the least favorable direction in this case. We focus on the same GP prior as in \citet{Van2009}, but slightly modify the prior for the inverse bandwidth $A$ to be a truncated $Ga(a_0,b_0)$ whose pdf $p(t)\propto t^{a_0-1}e^{-b_0t}I(t\geq t_0)$. Introducing this truncation is  for technical simplicity and will not sacrifice the adaptivity of the prior. The following result is an adaptive version of Theorem \ref{thm:4a,3a} under independent prior (PI).

\begin{theorem}\label{thm:3a'}
Let $X_i=(U_i,V_i,Y_i)\in\bbR^p\times\bbR^d\times\bbR$, $i=1,\ldots,n$, be $n$ observations from the partially linear model \eqref{eq:plmf}. Consider the independent prior (PI) with the above adaptive $\Pi_H$. Assume the following conditions:
\begin{enumerate}
  \item $\eta_0$ is H\"{o}lder $\alpha$-smooth, where $\alpha>d/2$;
    \item The information matrix $\widetilde{I}_0=P_0(U-E[U|V])^{\otimes 2}$ is invertible;
  \item The least favorable direction $h^\ast(v)=E[U|V=v]$ belongs to the RKHS $\mathbb{H}^{t_0}$, where $t_0$ is the truncation parameter in the above truncated $Ga(a_0,b_0)$.
\end{enumerate}
Then the following second order BvM theorem holds:
\begin{equation*}
    \sup_A\big|\Pi(\theta\in A|X^{(n)})-N_p\big(\theta_0+n^{-1/2}\widetilde{\Delta}_n,(n\widetilde{I}_0^{-1})\big)(A)\big|= O_{P_0}(\sqrt{n}\rho_n^2\log n)=O_{P_0}\big(n^{-\frac{\alpha-d/2}{2\alpha+d}}(\log n)^{2d+3}\big).
\end{equation*}
\end{theorem}

We point out that this theorem requires a strong constraint on the least favorable direction $h^\ast(v)$, i.e., Condition (iii). In fact, a sufficient condition for a function $f$ to belong to $\mathbb{H}^{t_0}$ is that it has a Fourier transform $\widehat{f}$ satisfying
\begin{align*}
    \int_{\bbR^d}|\widehat{f}(\lambda)|^2e^{c|\lambda|^2/t_0^2}d\lambda\leq\infty,
\end{align*}
for some $c>0$ \citep{Van2009}. This condition implies the infinite differentiability of $h^\ast$ and imposes a strong restriction---for example, it fails even with constant and polynomial functions. On the other hand, we have to admit that Condition (iii) is only a sufficient condition although our empirical results indicate that a similar condition is necessary.


\subsection{GPLM: Assumptions and LFS Lemma}
\label{SectionA.1}

\begin{assumption}\label{A.1}
\begin{enumerate}[(a)]
\item\label{assump:A1:a} There exists some positive constant $C_0$ such that $E_0(\exp(t|W|/C_0)|T)\leq C_0e^{C_0t^2}$, for all $t>0$, i.e. $W=Y-m_0(T)$ is sub-Gaussian.
\item\label{assump:A1:b} There exist positive constants $C_1$, $C_2$, $C_3$ and $C_4$ such that: 1. $1/C_1 \leq V(s) \leq C_1$ for all $s\in F(\bbR)$; 2.
    $1/C_2 \leq |l(\xi)| \leq C_2$ for all $\xi\in \bbR$;
    3. $|l(\xi)-l(\xi_0)| \leq C_3|\xi-\xi_0|$ for all $|\xi-\xi_0|\leq \eta_0$; 4. $|f(\xi)-f(\xi_0)| \leq C_4|\xi-\xi_0|$ for all $|\xi-\xi_0|\leq \eta_0$.
\end{enumerate}
\end{assumption}
The assumption that $V$ and $l$ are both bounded could be restrictive and
can be removed in many cases, such as the binary logistic regression model, by applying empirical process arguments similar to those in Section 7 of \cite{mammen1997}.
Under Assumption \ref{A.1}(\ref{assump:A1:b}), the following lemma describes the least favorable curve for the class of GPLM.
\begin{Lemma}\label{le:glfc}
Suppose Assumption \ref{A.1}(\ref{assump:A1:b}) is met. Then the least favorable curve $\eta^*(\theta)$, defined as the minimizer $\eta$ of
\begin{align*}
    E_0\log (p_{\theta_0,\eta_0}/p_{\theta,\eta})=E_0\int_{m_{\theta_0,\eta_0}(T)}
    ^{m_{\theta,\eta}(T)}\frac{(Y-s)}{V(s)}ds=E_0\int_{m_{\theta_0,\eta_0}(T)}
    ^{m_{\theta,\eta}(T)}\frac{(m_{\theta_0,\eta_0}(T)-s)}{V(s)}ds
\end{align*}
as a function of $\theta$, takes the following expression
\begin{align}
    &\eta^*(\theta)=\eta_0+(\theta-\theta_0)^Th^\ast(V)+O(|\theta-\theta_0|^2),\ \text{as } |\theta-\theta_0|\to0,\label{eq:glfc}\\
    &\text{with }\quad h^\ast(v)=-\frac{E_0\big[Uf_0(T)l_0(T)| V=v\big]}{E_0\big[f_0(T)l_0(T) | V=v\big]}.\label{eq:h}
\end{align}
\end{Lemma}
Equation~\eqref{eq:glfc} provides a local expansion of the least favorable curve defined in \eqref{eq:KL}, which is enough for our purpose, since the posterior of $\theta$ is expected to concentrate in a $\sqrt{n}$-neighborhood of $\theta_0$.
\begin{proof}[Proof of Lemma \ref{le:glfc}]
By Assumption \ref{A.1}(\ref{assump:A1:b}), for any $(\theta,\eta)$, we have
\begin{align*}
    E_0\log (p_{\theta,\eta}/p_{\theta_0,\eta_0})
    \leq&-C_1^{-1}E_0\big(m_{\theta,\eta}(T)-m_{\theta_0,\eta_0}(T)\big)^2\\
    \leq&-(C_1^{2}C_2)^{-1}E_0|g_{\theta,\eta}(T)-g_0(T)|^2\\
    \leq&-2(C_1^{2}C_2)^{-1}\big(|\theta-\theta_0|^2+E_0|\eta-\eta_0|^2\big),
\end{align*}
where the first line follows since $V(s)\leq C_1$, the second line follows by the fact that $|f(\xi)|=|l(\xi)|\cdot|V(F(\xi))|\in[1/(C_1C_2),C_1C_2]$ and the third line follows by the assumption that $U\in[0,1]^p$. Similarly, we have
\begin{align*}
    E_0\log (p_{\theta,\eta}/p_{\theta_0,\eta_0})
    \geq&-C_1^{2}C_2\,E_0\big((\theta-\theta_0)^TU+\eta(V)-\eta_0(V)\big)^2.
\end{align*}
Let $\bar{\eta}(\theta)(v)=\eta_0(v)-(\theta-\theta_0)^TE[U|V=v]$. Then by definition of $\eta^*(\theta)$, we have
\begin{align*}
 E_0\log (p_{\theta,\eta^*(\theta)}/p_{\theta_0,\eta_0})\geq  E_0\log (p_{\theta,\bar{\eta}(\theta)}/p_{\theta_0,\eta_0}).
\end{align*}
Combining the above inequalities, we obtain
\begin{align*}
    &-2(C_1^{2}C_2)^{-1}\big(|\theta-\theta_0|^2+E_0|\eta^*(\theta)-\eta_0|^2\big)\\
    \geq& -C_1^{2}C_2E_0\big((\theta-\theta_0)^TU+\bar{\eta}(\theta)(V)-\eta_0(V)\big)^2\\
    =&-C_1^{2}C_2E_0(U-E[U|V])^2|\theta-\theta_0|^2,
\end{align*}
which implies
\begin{align}\label{eq:dif}
    \eta^*(\theta)-\eta_0=O(|\theta-\theta_0|).
\end{align}

For an arbitrary function $h(V):\bbR^d\to\bbR^p$ with $\|h\|_{\infty}<\infty$, consider
 \begin{align*}
   \widehat{g}_{\theta,\eta^*(\theta),t}= g_{\theta,\eta^*(\theta)}+th,
\end{align*}
for $t$ in a neighborhood of $0$. The optimality of $g_{\theta,\eta^\ast(\theta)}$ implies that
\begin{align*}
    0=E_0\big[(Y-F(g_{\theta,\eta^*(\theta)}))l(g_{\theta,\eta^*(\theta)})h(V)\big]
    =E_0\big\{E_0\big[(F(g_{\theta_0,\eta_0})-F(g_{\theta,\eta^*(\theta)}))l(g_{\theta,\eta^*(\theta)})\big|V\big]h(V)\big\}.
\end{align*}
Since the above equality holds for any $h$, we have
\begin{align*}
    E_0\big[(F(g_{\theta_0,\eta_0})-F(g_{\theta,\eta^*(\theta)}))l(g_{\theta,\eta^*(\theta)})\big|V=v\big]=0, \ a.s.
\end{align*}
The last display, equation~\eqref{eq:dif} and Assumption \ref{A.1}\,(\ref{assump:A1:b}) together, implies
\begin{align*}
    E_0\big[f_0(T)l_0(T)(\theta-\theta_0)^TU\big|V=v\big]+E_0\big[f_0(T)l_0(T)\big|V=v\big]
    (\eta^*(\theta)-\eta_0)(v)=O(|\theta-\theta_0|^2), a.s.
\end{align*}
This gives us
\begin{align*}
    \eta^*(\theta)(v)=\eta_0(v)-(\theta-\theta_0)h^\ast(v)+O(|\theta-\theta_0|^2),\ \text{as } |\theta-\theta_0|\to0,
\end{align*}
with $h^\ast$ defined by \eqref{eq:h}.
\end{proof}

\subsection{Proofs of Theorem \ref{thm:MAIN1}}
\label{se:main1p}
Let $B_n:\,=\{|\theta-\theta_0|\leq M\epsilon_n,\ \eta\in\mathcal{H}_n\}$, where $M$ is a sufficiently large constant. Then we have $\Pi(B_n|X^{(n)})=1-O_{P_0}(\delta_n)$ for $M\geq 1$ by Assumption~\ref{A.2}. For any measurable $A\subset\Theta$,
\begin{align*}
    \big|\Pi(\theta\in A|X^{(n)},B_n)-\Pi(\theta\in A|X^{(n)})\big|=&\,
    \bigg|\frac{\Pi(\theta\in A,B_n^c|X_n)-\Pi(\theta\in A|X^{(n)})\big[1-\Pi(B_n|X^{(n)})\big]}{\Pi(B_n|X^{(n)})}\bigg|\\
    \leq &\,2\big|1-\Pi(B_n|X^{(n)})\big|\,\big/\, \Pi(B_n|X^{(n)})=O_{P_0}(\delta_n).
\end{align*}
Taking the supremum of $A$ over all measurable subsets of $\Theta$, we obtain
\begin{align*}
     \sup_A\big|\Pi(\theta\in A|X^{(n)},B_n)-\Pi(\theta\in A|X^{(n)})\big|=O_{P_0}(\delta_n).
\end{align*}
Therefore, it remains to show that
\begin{equation}\label{eq:hosemiBvMb}
    \sup_A\big|\Pi(\theta\in A|X_1,\ldots,X_n,B_n)-N_k\big(\widetilde{\Delta}_n,(n\widetilde{I}_{0})^{-1}\big)(A)\big|= O_{P_0}[R_n(n^{-1/2}\log n)],
\end{equation}
where
\begin{equation}\label{semipos}
\begin{aligned}
    \Pi(\theta\in A|X_1,\ldots,X_n,B_n)=&\int_{A\cap\{|\theta-\theta_0|\leq M\epsilon_n\}}\frac{\widetilde{S}_n(\theta)}{\widetilde{S}_n(\theta_0)}d\Pi_{\Theta}(\theta)\bigg/\int_
    {|\theta-\theta_0|\leq M\epsilon_n}\frac{\widetilde{S}_n(\theta)}{\widetilde{S}_n(\theta_0)}d\Pi_{\Theta}(\theta).
\end{aligned}
\end{equation}

Recall the definition of $\widetilde{\Delta}_n$ by \eqref{eq:esf}.
Since the pdf of a normally distributed random variable with mean $\theta_0+n^{-1/2}\widetilde{\Delta}_n$ and covariance matrix $(n\widetilde{I}_{0})^{-1}$ evaluated at $\theta$ is proportional to
\begin{align*}
   \exp\bigg\{(\theta-\theta_0)^T
    \sum_{i=1}^n\widetilde{\ell}_{0}(X_i)
    -\frac{n}{2}(\theta-\theta_0)^T\widetilde{I}_{0}(\theta-\theta_0)
   \bigg\},
\end{align*}
it suffices to prove
\begin{equation}\label{eq:thmm}
\begin{aligned}
    \bigg|\int_{A}&\exp\bigg\{(\theta-\theta_0)^T
    \sum_{i=1}^n\widetilde{\ell}_{0}(X_i)
    -\frac{n}{2}(\theta-\theta_0)^T\widetilde{I}_{0}(\theta-\theta_0)\bigg\}d\theta
    -\int_{A\cap\{|\theta-\theta_0|\leq M\epsilon_n\}}\frac{\widetilde{S}_n(\theta)}{\widetilde{S}_n(\theta_0)}d\Pi(\theta)
    \bigg|\\
    =& O_{P_0}[R_n(n^{-1/2}\log n)]
    \int_{\Theta}\exp\bigg\{(\theta-\theta_0)^T
    \sum_{i=1}^n\widetilde{\ell}_{0}(X_i)
    -\frac{n}{2}(\theta-\theta_0)^T\widetilde{I}_{0}(\theta-\theta_0)\bigg\}d\theta.
\end{aligned}
\end{equation}
In fact, one can plug in the above equation with $A=A$ and $A=\Theta$ respectively, and then simple algebra leads to \eqref{eq:hosemiBvMb}.

Since $n\epsilon_n^2\gtrsim-\log R_n(n^{-1/2}\log n)\to\infty$ and $\sum_{i=1}^n\widetilde{\ell}_{0}=O_{P_0}(\sqrt{n})$, by choosing $M$ sufficiently large we have
\begin{equation}\label{eq:thm1}
\begin{aligned}
\bigg|&\int_{A\cap\{|\theta-\theta_0|>M\epsilon_n\}}\exp\bigg\{(\theta-\theta_0)^T
    \sum_{i=1}^n\widetilde{\ell}_{0}(X_i)
    -\frac{n}{2}(\theta-\theta_0)^T\widetilde{I}_{0}(\theta-\theta_0)\bigg\}d\theta\\
    =& O_{P_0}[R_n(n^{-1/2}\log n)]\int_{\Theta}\exp\bigg\{(\theta-\theta_0)^T
    \sum_{i=1}^n\widetilde{\ell}_{0}(X_i)
    -\frac{n}{2}(\theta-\theta_0)^T\widetilde{I}_{0}(\theta-\theta_0)\bigg\}d\theta.
\end{aligned}
\end{equation}
By a subsequence argument, Assumption~\ref{A.3} implies
\begin{align}\label{eq:ilanc}
\sup_{|\theta-\theta_0|\leq M\epsilon_n}\Big|\log\frac{\widetilde{S}_n(\theta)}{\widetilde{S}_n(\theta_0)}-(\theta-\theta_0)^T
    \sum_{i=1}^n\widetilde{\ell}_{0}(X_i)
    +\frac{n}{2}(\theta-\theta_0)^T\widetilde{I}_{0}
    (\theta-\theta_0)\Big|\big/R_n(|\theta-\theta_0|)=O_{P_0}(1).
\end{align}

For every $\theta$ such that $|\theta-\theta_0|<Mn^{-1/2}\log n$ with $M$ sufficiently large, the above analysis implies that
\begin{equation}\label{eq:thm2}
\begin{aligned}
    &\bigg|\exp\bigg\{(\theta-\theta_0)^T
    \sum_{i=1}^n\widetilde{\ell}_{0}(X_i)
    -\frac{n}{2}(\theta-\theta_0)^T\widetilde{I}_{0}(\theta-\theta_0)\bigg\}
    -\frac{\widetilde{S}_n(\theta)}{\widetilde{S}_n(\theta_0)}\bigg|\\
    \leq&\exp\bigg\{(\theta-\theta_0)^T
    \sum_{i=1}^n\widetilde{\ell}_{0}(X_i)
    -\frac{n}{2}(\theta-\theta_0)^T\widetilde{I}_{0}(\theta-\theta_0)\bigg\}
    \big|\exp\big\{O_{P_0}[R_n(n^{-1/2}\log n)]\big\}-1\big|\\
    =& O_{P_0}[R_n(n^{-1/2}\log n)]
    \exp\bigg\{(\theta-\theta_0)^T
    \sum_{i=1}^n\widetilde{\ell}_{0}(X_i)
    -\frac{n}{2}(\theta-\theta_0)^T\widetilde{I}_{0}(\theta-\theta_0)\bigg\},
\end{aligned}
\end{equation}
where the last step follows since $R_n(n^{-1/2}\log n)\to 0$.

For every $\theta$ such that $Mn^{-1/2}\log n\leq|\theta-\theta_0|<M\epsilon_n$ with $M$ sufficiently large, we have by Assumption~\ref{A.3} and the invertibility of $\widetilde{I}_{0}$ that $R_n(|\theta-\theta_0|)/[n(\theta-\theta_0)^T\widetilde{I}_{0}(\theta-\theta_0)]=o(1)$. Combining this fact and the last display, we obtain
\begin{equation}\label{eq:thm3}
\begin{aligned}
    &\bigg|\int_{A\cap\{Mn^{-1/2}\log n\leq|\theta-\theta_0|<M\epsilon_n\}}\exp\bigg\{(\theta-\theta_0)^T
    \sum_{i=1}^n\widetilde{\ell}_{0}(X_i)
    -\frac{n}{2}(\theta-\theta_0)^T\widetilde{I}_{0}(\theta-\theta_0)\bigg\}d\theta\\
    &\qquad\qquad\qquad\qquad\qquad\qquad\qquad\qquad-\int_{A\cap\{Mn^{-1/2}\log n\leq|\theta-\theta_0|< M\epsilon_n\}}\frac{\widetilde{S}_n(\theta)}{\widetilde{S}_n(\theta_0)}d\Pi(\theta)
    \bigg|\\
    =& O_{P_0}(1)\int_{|\theta-\theta_0|>Mn^{-1/2}\log n}\exp\bigg\{(\theta-\theta_0)^T
    \sum_{i=1}^n\widetilde{\ell}_{0}(X_i)
    -\frac{n}{4}(\theta-\theta_0)^T\widetilde{I}_{0}(\theta-\theta_0)\bigg\}d\theta\\
    =& O_{P_0}(e^{-Mc(\log n)^2})\int_{\Theta}\exp\bigg\{(\theta-\theta_0)^T
    \sum_{i=1}^n\widetilde{\ell}_{0}(X_i)
    -\frac{n}{8}(\theta-\theta_0)^T\widetilde{I}_{0}(\theta-\theta_0)\bigg\}d\theta\\
    =&O_{P_0}[R_n(n^{-1/2}\log n)]\int_{\Theta}\exp\bigg\{(\theta-\theta_0)^T
    \sum_{i=1}^n\widetilde{\ell}_{0}(X_i)
    -\frac{n}{2}(\theta-\theta_0)^T\widetilde{I}_{0}(\theta-\theta_0)\bigg\}d\theta,
\end{aligned}
\end{equation}
for $M$ sufficiently large, where $c>0$ is a constant only depending on $\widetilde{I}_{0}$ and the last step follows by the fact that $\int\exp\{at-bt^2\}dt\asymp b^{-1/2}$ for $b\gg \min(a,1)$.

Finally, \eqref{eq:thm1} together with \eqref{eq:thm2} and \eqref{eq:thm3} implies \eqref{eq:thmm}.

\subsection{Proof of Corollary \ref{coro:1}}
For each $s=1,\ldots,p$, taking $A=\bbR\times\cdots\times A_s\times\cdots\times\bbR$ in \eqref{eq:sBvM}, where the $s$-th component is $A_s$ and the rest are $\bbR$, we obtain
\begin{equation}\label{eq:1d}
    \sup_{A_s\subset\bbR}\big|\Pi(\theta_s\in A_s|X_1,\ldots,X_n)-N_p\big(\theta_{0,s}+n^{-1/2}\widetilde{\Delta}_{n,s},n^{-1}\widetilde{I}_{0}^{ss}\big)(A_s)\big|= O_{P_0}(S_n),
\end{equation}
where $\widetilde{\Delta}_{n,s}$ is the $s$th component of $\widetilde{\Delta}_{n}$ and
$\widetilde{I}_{0}^{ss}$ the $(s,s)$-th element of the matrix $\widetilde{I}_{0}^{-1}$. Let $\widehat{\theta}^B_{n,s}$ be the median of the marginal posterior distribution of $\theta_s$. Then taking $A_s=(-\infty,\widehat{\theta}^B_{n,s})$ in the above formula yields
\begin{align*}
    \big|\Phi\big(n^{1/2}(\widetilde{I}_{0}^{ss})^{-1/2}(\widehat{\theta}^B_{n,s}-\theta_{0,s}-n^{-1/2}\widetilde{\Delta}_{n,s})\big)-1/2\big|=O_{P_0}(S_n),
\end{align*}
where $\Phi$ is the cdf of the standard normal distribution. By the continuity of $\Phi^{-1}$, we have
\begin{align*}
    n^{1/2}(\widetilde{I}_{0}^{ss})^{-1/2}(\widehat{\theta}^B_{n,s}-\theta_{0,s}-n^{-1/2}\widetilde{\Delta}_{n,s})=O_{P_0}(S_n),
\end{align*}
which proves the claimed result.

\subsection{Proof of Corollary \ref{coro:2}}
Recall that $\widetilde{I}_{0}^{ss}$ is the $(s,s)$-th element of $\widetilde{I}_{0}^{-1}$.
By choosing $A_s=(-\infty, \widehat{q}_{s,\alpha})$ in \eqref{eq:1d} and the definition of $\widehat{q}_{s,\alpha}$, we have
\begin{align*}
    \big|\Phi\big(n^{1/2}(\widetilde{I}_{0}^{ss})^{-1/2}(\widehat{q}_{s,\alpha}-\theta_{0,s}-n^{-1/2}\widetilde{\Delta}_{n,s})\big)-\alpha\big|=O_{P_0}(S_n),
\end{align*}
which implies $\widehat{q}_{s,\alpha}=\theta_{0,s}+n^{-1/2}\widetilde{\Delta}_{n,s}+n^{-1/2}(\widetilde{I}_{0}^{ss})^{1/2}z_{\alpha}+n^{-1/2}\,O_{P_0}(S_n)$, where $z_{\alpha}$ denotes the $\alpha$-th quantile of a standard normal distribution. This  completes the proof of the claimed result.

\subsection{Proof of Lemma \ref{thm:2}}
With the definition of $\widetilde{S}_n$ and the conditions in the lemma, we have
\begin{align*}
    \widetilde{S}_n(\theta)=&\int_{\mathcal{H}_n}\exp\{l_n(\theta,\eta)-l_n(\theta_0,\eta_0)\}
    d\Pi^{\theta}_H(\eta)\\
    =&\exp\bigg\{\sqrt{n}(\theta_n-\theta_0)^T\widetilde{g}_n
-\frac{1}{2}n(\theta_n-\theta_0)^T\widetilde{I}_{0}(\theta_n-\theta_0)\\
&+O_{P_0}[G_n(|\theta-\theta_0|)]\bigg\}
    \int_{\mathcal{H}_n}\exp\{l_n(\theta_0,\eta-\Delta\eta(\theta))-l_n(\theta_0,\eta_0)\}
    d\Pi^{\theta}_H(\eta)\\
     =&\exp\bigg\{\sqrt{n}(\theta_n-\theta_0)^T\widetilde{g}_n
-\frac{1}{2}n(\theta_n-\theta_0)^T\widetilde{I}_{0}(\theta_n-\theta_0)\\
&+O_{P_0}[R_n(|\theta-\theta_0|)])\bigg\}\widetilde{S}_n(\theta_0),
\end{align*}
where the second line follows by condition (A1) and the last step follows by condition (A2).
Finally, the ILAN in Assumption 2 follows by taking logarithms of both sides of the above equaility.

\subsection{Proof of Lemma \ref{eq:pinp}}
Under (A5), we have
\begin{align}
   \frac{\int_{\mathcal{H}_n-(\theta_n-\theta_0)^Th_n}e^{l_n(\theta_0,\eta)}
    d\Pi_H(\eta)}{\int_{\mathcal{H}_n}e^{l_n(\theta_0,\eta)}
    d\Pi_H(\eta)}
    =&\frac{\int_{\mathcal{H}_n-(\theta_n-\theta_0)^Th_n}e^{l_n(\theta_0,\eta)}
    d\Pi_H(\eta)}{\int_{\mathcal{H}}e^{l_n(\theta_0,\eta)}
    d\Pi_H(\eta)}\cdot\frac{\int_{\mathcal{H}}e^{l_n(\theta_0,\eta)}
    d\Pi_H(\eta)}{\int_{\mathcal{H}_n}e^{l_n(\theta_0,\eta)}
    d\Pi_H(\eta)}\nonumber\\
    =&\frac{\Pi_H(\mathcal{H}_n-(\theta_n-\theta_0)^Th_n|X_1,\ldots,X_n)}
    {\Pi_H(\mathcal{H}_n|X_1,\ldots,X_n)}
    =1+O_{P_0}(\delta_n).\label{eq:argu}
\end{align}
By applying a change of variables $\widetilde{\eta}=\eta-(\theta_n-\theta_0)^Th_n$ in the numerator in (A2) and using (A4), we can obtain
\begin{align*}
    \int_{\mathcal{H}_n}e^{l_n(\theta_0,\eta-\Delta\eta(\theta_n))}    d\Pi_{\mathcal{H}}(\eta)=&\int_{\mathcal{H}_n-(\theta_n-\theta_0)^Th_n} e^{l_n(\theta_0,\widetilde{\eta})}
    f_{n}(\widetilde{\eta})d\Pi_{\mathcal{H}}(\widetilde{\eta})\nonumber\\
    &\qquad\qquad\cdot\big\{1+O_{P_0}[\bar{G}_n(|\theta_n-\theta_0|)]\big\},
\end{align*}
which combined with (A6) yields
\begin{align}
\label{eq:argu2}
   \frac{\int_{\mathcal{H}_n} e^{l_n(\theta_0,\eta-\Delta\eta(\theta_n))}d\Pi_{\mathcal{H}}(\eta)}{\int_{\mathcal{H}_n-(\theta_n-\theta_0)^Th_n}e^{l_n(\theta_0,\eta)}
    d\Pi_H(\eta)} = 1+O_{P_0}[\bar{G}_n(|\theta_n-\theta_0|)].
\end{align}
Finally, combining \eqref{eq:argu} and \eqref{eq:argu2} implies (A2).

\subsection{Proof of Lemma \ref{le:dprior}}
Applying a change of variables $\widetilde{\eta}=\eta-(\theta_n-\theta_0)^T\widehat{h}_n$, we obtain
\begin{align*}
    &\int_{\mathcal{H}_n}e^{l_n(\theta_0,\eta-\Delta\eta(\theta_n))}
    d\Pi_H^{\theta_n}(\eta)\\
    =&\int_{\mathcal{H}_n-(\theta_n-\theta_0)^T\widehat{h}_n}
    e^{l_n(\theta_0,\widetilde\eta-\Delta\eta(\theta_n)+(\theta_n-\theta_0)^T\widehat{h}_n)}
    d\Pi^{\theta_n}_{H,\cdot-(\theta_n-\theta_0)^T\widehat{h}_n}(\widetilde\eta)\\
    =&\int_{\mathcal{H}_n-(\theta_n-\theta_0)^T\widehat{h}_n}
    e^{l_n(\theta_0,\widetilde\eta-\Delta\eta(\theta_n)+(\theta_n-\theta_0)^T\widehat{h}_n)}
    d\Pi^{\theta_0}_H(\widetilde\eta)\\
    =&\big(1+O_{P_0}[\bar{G}_n(\max\{|\theta-\theta_0|,n^{-1/2}\log n\})]\big)\int_{\mathcal{H}_n-(\theta_n-\theta_0)^T\widehat{h}_n}
    e^{l_n(\theta_0,\widetilde\eta)}
    d\Pi^{\theta_0}_H(\widetilde\eta)\\
    =&\big(1+O_{P_0}[\widetilde{G}_n(\max\{|\theta-\theta_0|,n^{-1/2}\log n\})]\big)\int_{\mathcal{H}_n}e^{l_n(\theta_0,\widetilde\eta)}
    d\Pi^{\theta_0}_H(\widetilde\eta),
\end{align*}
where the second step follows by the definition of the prior (PD), the third step by (A4), and the last step by \eqref{eq:argu}.

\subsection{Proof of Theorem \ref{thm:4a,3a}}
For readers' convenience, we state the maximal inequality for sub-Gaussian random variables \citep[Corollary 2.2.8]{Van1996} which is extensively applied in our examples.
\begin{Lemma}
\label{Lemma:MI}
Let $\{W_t: \, t\in T\}$ be a separable sub-Gaussian process and $d$ be a semimetric on the index set $T$ defined by $d(s,t)=\sigma(W_s-W_t)$. Then for every $\delta>0$ and $x>0$,
\begin{eqnarray}
P\Big(\sup_{d(s,t)\leq \delta} |W_s-W_t|\geq x\Big)&\leq 2\exp\Big\{-x^2\Big/ \Big(K \int_0^{\delta} \sqrt{\log N(\epsilon, T, d)}\, d\epsilon\Big)^2\Big\},\label{inter1}\\
E\sup_{d(s,t)\leq \delta} |W_s-W_t| &\leq K \int_0^{\delta} \sqrt{\log N(\epsilon, T, d)}\, d\epsilon,\label{inter2}
\end{eqnarray}
for a universal constant $K$.
\end{Lemma}

We consider the independent prior and the dependent prior separately.
\paragraph{Independent prior:}

\underline{Verification of Assumption \ref{A.2}}: We apply Lemma~\ref{le:pcr2} here. Let $\bbN$ denote the set of natural numbers and $\bbN_0 = \bbN \cup \{0\}$. For any $d$ dimensional multi-index $a = (a_1, \ldots, a_d) \in \bbN_0^d$, define $|a| = a_1 + \cdots + a_d$ and let $D^a$ denote the mixed partial derivative operator $\partial^{|a|} / \partial x_1^{a_1} \cdots \partial x_d^{a_d}$. For any real number $b$, let $\lfloor b \rfloor$ denote the largest integer strictly smaller than $b$. The H\"older class $\mathcal{C}^{\gamma}([0,1]^d)$ is defined as
the set of all $d$-variate $k=\lfloor \gamma\rfloor$ times differentiable functions $f$ on $[0,1]^d$ such that:
\[
\|f\|_{\mathcal{C}^{\gamma}}:\,=\max_{|\beta|\leq k}\sup_{x\in[0,1]^d}|D^{\beta}f(x)|+\max_{|\beta|=k}\sup_{x\neq y}\frac{|D^{\beta}(x)-D^{\beta}(y)|}{|x-y|^{\gamma-k}}<\infty.
\]
We use $\mathcal{C}^{\gamma}_1$ to denote the unit ball in $\mathcal{C}^{\gamma}$ under the norm $\|\cdot\|_{\mathcal{C}^{\gamma}}$.

We choose the sieve $\mathcal{F}_n$ as $\mathcal{F}_n^{\theta}\oplus\mathcal{F}_n^{\eta}$, with
\begin{align}
    &\mathcal{F}_n^{\theta}=[-c\sqrt{n}, c\sqrt{n}]^p\ \text{ and }\ \mathcal{F}_n^{\eta}=\rho_n \mathcal{C}^{\alpha}_1+M_n\mathbb{H}_1^{a_n},\label{eq:isieves}
\end{align}
with $c$ a constant sufficiently large, $\rho_n=n^{-\alpha/(2\alpha+d)}(\log n)^{d+1}$, $a_n=n^{1/(2\alpha+d)}$, and $M_n$ some constant to be determined later.
The second term $M_n\mathbb{H}_1^{a_n}$ in the sieve construction for $\eta$ borrows the ideas from \cite{Van2008b} and the first term $\rho_nC_1^{\alpha}$ from \cite{Jonge2013}. We remark that in \cite{Van2008b} the first term in their sieve construction ($B_n$ on page 20) is a multiple of $B_1:\,=\{f\in L_2([0,1]^d):\|f\|_{\infty}\}$, causing the functions in $\mathcal{F}_n^{\eta}$ to be non-differentiable. As a consequence, the $\epsilon$-covering entropy of their sieve can not be properly bounded when $\epsilon<\rho_n$ as in our proof (see \eqref{eq:ice} below).

By Lemma 4.5 in \cite{Van2009}, for a fixed scaling parameter $a$ and any $\epsilon<1/2$, we have the following upper bound on the covering entropy of the unit ball in the RKHS $\mathbb{H}^a$,
\begin{align}\label{eq:GPCE}
    \log N(\epsilon,\mathbb{H}_1^a,\|\cdot\|_{\infty})\leq K_1a^d\bigg(\log\frac{1}{\epsilon}\bigg)^{1+d},
\end{align}
where $K_1$ is some universal constant.
For squared exponential kernel, all elements in $\mathbb{H}_1^a$ are infinitely differentiable.
Consequently, by slightly modifying their proof, the sup-norm in the above result can be generalized to the $\|\cdot\|_{\mathcal{C}^{\gamma}}$-norm: for any smoothness index $\gamma>0$,
\begin{align*}
    \log N(\epsilon,\mathbb{H}_1^a,\|\cdot\|_{\mathcal{C}^{\gamma}})\leq K_1a^d\bigg(\log\frac{a^{\gamma}}{\epsilon}\bigg)\bigg(\log\frac{1}{\epsilon}\bigg)^{d}.
\end{align*}
Then by the relationship between the small ball probability of a Gaussian process and the covering entropy of the unit ball in the associated RKHS \citep{Li1999}, we can obtain by following the proof of Lemma 4.6 in \cite{Van2009} that for any $\gamma>0$,
\begin{align}\label{eq:GPsp}
    -\log \Pi(\|W^a\|_{\mathcal{C}^{\gamma}}\leq \epsilon)\leq Ka^d\bigg(\log\frac{a}{\epsilon}\bigg)^{1+d}.
\end{align}
Denote the right hand side of the above by $\phi_0^a(\epsilon)$. Note that the above also holds when the $\|\cdot\|_{\mathcal{C}^{\gamma}}$ norm is replaced with the sup-norm by applying inequality~\eqref{eq:GPCE} instead. Then by Borell's inequality \citep{Van2008},
\begin{align}\label{eq:GPsp2}
    \Pi(W^a\notin M\mathbb{H}_1^a+\epsilon \mathcal{C}^{\alpha}_1)\leq 1-\Phi
    (\Phi^{-1}(e^{-\phi_0^a(\epsilon)})+M),
\end{align}
where $\Phi$ is the c.d.f.\! of the standard normal distribution.
Note that for $M>4\sqrt{\phi_0^{a}(\epsilon)}$, the right hand side of the last display is bounded by $e^{-M^2/8}$.

By applying the inequality~\eqref{eq:GPCE} with $a=a_n$, we can obtain the following bound on the $\epsilon$-covering entropy of the sieve $\mathcal{F}_n$ for any $\epsilon>0$,
\begin{align}\label{eq:ice}
    \log N(4\epsilon,\mathcal{F}_n,\|\cdot\|_{\infty})\leq
    K_2 n\rho_n^2(\log n)^{-(1+d)}  \bigg(\log\bigg(\frac{n}{\epsilon}\bigg)\bigg)^{1+d}
    +K_2\bigg(\frac{\rho_n}{\epsilon}\bigg)^{d/\alpha}+c\log\bigg(\frac{n}{\epsilon}\bigg),
\end{align}
where we have used the fact that the covering entropy of $\mathcal{C}^{\alpha}_1([0,1]^d)$ satisfies
$\log N(\epsilon,\mathcal{C}^{\alpha}_1,\|\cdot\|_{\infty})\leq K_2\, \epsilon^{-d/\alpha}$ and $K_2$ is some constant.
By choosing $M_n=c_1n\rho_n^2$ with $c_2$ sufficiently large so that $M_n>4\sqrt{\phi_0^{a_n}(\rho_n)}$ and applying inequality~\eqref{eq:GPsp2}, we have the following complement probability bound on $\mathcal{F}_n$ with some constant $c_2>0$,
\begin{align}\label{eq:icp}
    \Pi(\mathcal{F}_n^c)\leq\exp(-c_2n\rho_n^2).
\end{align}
Therefore, sieve $\mathcal{F}_n$ satisfies condition~\emph{a} and condition~\emph{b} in Lemma \ref{le:pcr2} with $\xi_n=\rho_n$.
Next we verify condition~\emph{c} in Lemma~\ref{le:pcr2}.
For the partially linear model, we have,
\begin{align*}
K(P_{\theta_0,\eta_0}^{(n)},P_{\theta,\eta}^{(n)})=&E_0\big\{
    \log(dP_{\theta_0,\eta_0}^{(n)}/dP_{\theta,\eta}^{(n)})\big\}\\
    =&\frac{1}{2}\sum_{i=1}^n[(\theta-\theta_0)U_i+(\eta-\eta_0)(V_i)]^2,
\end{align*}
and
\begin{align*}
    V_{2}(P_{\theta_0,\eta_0}^{(n)},P_{\theta,\eta}^{(n)})=&E_0\Big\{
   \Big |\log(dP_{\theta_0,\eta_0}^{(n)}/dP_{\theta,\eta}^{(n)})-K(P_{\theta_0,\eta_0}^{(n)},
    P_{\theta,\eta}^{(n)})\Big|^2\Big\}\\
    =& E_0\Big\{
    \Big(\sum_{i=1}^nw_i\big[(\theta-\theta_0)U_i+(\eta-\eta_0)(V_i)\big]\Big)^2\Big|U^n,V^n\Big\}\\
    = & \sum_{i=1}^n[(\theta-\theta_0)U_i+(\eta-\eta_0)(V_i)]^2,
\end{align*}
where the last step follows by the fact that given $(U_i,V_i)$, the random variable $\sum_{i=1}^nw_i[(\theta-\theta_0)U_i+(\eta-\eta_0)(V_i)]$ follows a normal distribution with mean zero and variance $\big(\sum_{i=1}^n[(\theta-\theta_0)U_i+(\eta-\eta_0)(V_i)]^2\big)^{1/2}$. Therefore, for any $\epsilon>0$ we have
\begin{align*}
    B_n\big(P_0^{(n)},\epsilon\big)=&\big\{(\theta,\eta):
    K(P_{\theta_0,\eta_0}^{(n)}, P_{\theta,\eta}^{(n)})\leq n\epsilon^2,V_{2}(P_{\theta_0,\eta_0}^{(n)},P_{\theta,\eta}^{(n)})\leq n\epsilon^2\big\}\\
    =&\big\{(\theta,\eta): \|U^T(\theta-\theta_0)+\eta-\eta_0\|_n^2 \leq \epsilon^2\}.
\end{align*}
As a result, by applying inequality~\eqref{eq:GPsp} with the sup-norm and $\epsilon=\rho_n/2$,
we obtain that for the independent prior, there exists some constant $c_3$ such that
\begin{align}\label{eq:icpc}
    \Pi(B_n\big(P_0^{(n)},\rho_n)) \geq \Pi_{\Theta}(\|\eta-\eta_0\|_{\infty}\leq \rho_n/2)\cdot \Pi_{\mathcal{H}}(|\theta-\theta_0|\leq \rho_n/2) \geq \exp(-c_3n\rho_n^2).
\end{align}

Before applying Lemma \ref{le:pcr2}, we remark that although the average Hellinger metric $d_n$ used in Lemma \ref{le:pcr2} is equivalent to the empirical metric $\|\cdot\|_n$ only if the  class of regression functions is uniformly bounded, the argument in Section 7.7 of \cite{Ghosal2007} suggests that we may use $\|\cdot\|_n$ instead of $d_n$ throughout. Hence, the distance (between $(\theta,\eta)$ and $(\theta',\eta')$) is given by $\|U^T(\theta-\theta')+\eta-\eta'\|_n:\,= \sqrt{n^{-1}\sum_{i=1}^n\big(U_i^T(\theta-\theta')+\eta(V_i)-\eta'(V_i)\big)^2}$.
Therefore, by combining \eqref{eq:ice}, \eqref{eq:icp}, \eqref{eq:icpc} and Lemma~\ref{le:pcr2}, we can prove Assumption \ref{A.2} and conclude that
\begin{align}\label{eq:crplm1}
    \Pi\big\{\|U^T(\theta-\theta_0)+\eta-\eta_0\|_n\leq M\rho_n,\, \theta\in\mathcal{F}_n^{\theta},\,\eta\in\mathcal{F}_n^{\eta}\big|X_1,\ldots,X_n\big\}= 1- O_{P_0}(\delta_n),
\end{align}
where $M$ is a constant and $\delta_n=e^{-Cn\epsilon_n^2}$ for some $C>0$.

Next, we show that under Condition~(ii) in Theorem~\ref{thm:4a,3a}, \eqref{eq:crplm1} implies $\Pi\big(|\theta-\theta_0|\leq M\rho_n, \|\eta-\eta_0\|_n\leq M\rho_n\big|X^{(n)}\big)=1-O_{P_0}(\delta_n)$. Denote $I_n^2= n^{-1}\sum_{i=1}^n \big((\eta-\eta_0)(V_i)+(\theta-\theta_0)^TE[U_i|V_i]\big)^2$. We need to apply the following lemma, whose proof is provided in Subsection~\ref{sec:lem}.
\begin{Lemma}\label{LemmaCM}
Under the condition of the theorem, we have
\begin{align*}
  \sup_{\theta\in\mathcal{F}_n^{\theta},\,\eta\in\mathcal{F}_n^{\eta}} \frac{\big|\frac{1}{\sqrt{n}}\sum_{i=1}^n \big(U_i-E(U_i|V_i)\big)\cdot\big((\eta-\eta_0)(V_i)+(\theta-\theta_0)^TE(U_i|V_i)\big)\big|}{\sqrt{n}\rho_n I_n\log n \vee \sqrt{n}\rho_n^2}
  &=O_{P_0}(1).
\end{align*}
\end{Lemma}

By Lemma~\ref{LemmaCM}, we have that for any $\theta\in\mathcal{F}_n^{\theta}$ and $\eta\in\mathcal{F}_n^{\eta}$,
\begin{align*}
    &\frac{1}{n}\sum_{i=1}^n\big((\theta-\theta_0)^TU_i+(\eta-\eta_0)(V_i)\big)^2\\
    =&\,\frac{1}{n}\sum_{i=1}^n\bigg((\theta-\theta_0)^T(U_i-E[U_i|V_i])+\big((\eta-\eta_0)(V_i)
    +(\theta-\theta_0)^TE(U_i|V_i)\big)\bigg)^2\\
    \overset{(i)}{=}&(\theta-\theta_0)^T\big[P(U-E[U|V])^T(U-E[U|V])+O_{P_0}(n^{-1/2})\big](\theta-\theta_0)\\
    &\qquad\qquad\qquad\qquad\qquad\qquad\qquad\qquad+O_{P_0}(n^{-1/2})\cdot|\theta-\theta_0|\cdot (I_n\log n\vee\rho_n)+I_n^2\\
    \geq &\, (K_3+o_{P_0}(1))\cdot (|\theta-\theta_0|^2+ (I_n\vee\rho_n)^2),
\end{align*}
for some constant $K_3$, where in step (i) we have applied the central limit theorem for the sum $\sum_{i=1}^n(U_i-E[U_i|V_i])^T(U_i-E[U_i|V_i])$ and in the last step we have used the assumption that the matrix $P(U-E[U|V])^T(U-E[U|V])$ is invertible.
Combining the above with \eqref{eq:crplm1}, we obtain
\begin{align*}
    \Pi(|\theta-\theta_0|\leq M\rho_n|X_1,\ldots,X_n)=1-O_{P_0}(\delta_n).
\end{align*}
Again applying \eqref{eq:crplm1} and using the inequality $(a+b)^2\geq b^2/2-a^2$, we have that for $M$ sufficiently large,
\begin{align*}
    \Pi(\|\eta-\eta_0\|_n\leq M\rho_n|X_1,\ldots,X_n)=1-O_{P_0}(\delta_n).
\end{align*}
Combining the two above yields
$$
\Pi\big(|\theta-\theta_0|\leq M\rho_n, \|\eta-\eta_0\|_n\leq M\rho_n\big|X^{(n)}\big)=1-O_{P_0}(\delta_n).
$$
Therefore, if we define the localization sequence $\mathcal{H}_n=\{\eta\in\mathcal{F}_n^{\eta}:\|\eta-\eta_0\|_n\leq M\rho_n\}$, then the last display and Lemma \ref{le:pcr2} implies
\begin{align}\label{eq:con1}
    \Pi\big(|\theta-\theta_0|\leq M\rho_n,\eta\in\mathcal{H}_n\big|X^{(n)}\big)=1-O_{P_0}(\delta_n).
\end{align}
Note that with the enlargement procedure described after assumption (A5), the above still holds.

\underline{Verification of (A3)}: (A3) is true with $h^\ast(v)=-E[U|V=v]$.

\underline{Verification of (A1)}: We verify assumption (A1) with the above choice of $\mathcal{H}_n$. For the partially linear model, $\Delta\eta(\theta)=-(\theta-\theta_0)^TE[U|V]$. We use the notation $\mathbb{P}_n=n^{-1}\sum_{i=1}^n\delta_{X_i}$ to denote the empirical measure and $\mathbb{G}_n=n^{-1/2}\sum_{i=1}^n(\delta_{X_i}-P)$ the empirical process with respect to an i.i.d.\! sequence $\{X_i\}$.
For a partially linear model, we can express the log likelihood ratio by
\begin{align*}
    \log\frac{dP_{\theta,\eta+\Delta\eta(\theta)}}{dP_{\theta_0,\eta_0}}(X^{(n)})=&
    -\frac{1}{2}\sum_{i=1}^n\big[w_i-(\eta-\eta_0)(V_i)-(\theta-\theta_0)^T(U_i-E[U|V_i])\big]^2\\
    &+\frac{1}{2}\sum_{i=1}^n\big[w_i-(\eta-\eta_0)(V_i)\big]^2\\
    =&(\theta-\theta_0)^T\sum_{i=1}^n\widetilde{\ell}_{0}(X_i)-
    \frac{n}{2}(\theta-\theta_0)^T\widetilde{I}_0(\theta-\theta_0)+\frac{1}{2}\sqrt{n}(\theta-\theta_0)^2
    \mathbb{G}_n\big(U-E[U|V]\big)^2\\
    &-(\theta-\theta_0)^T\sum_{i=1}^n\big(U_i-E[U|V_i]\big)(\eta-\eta_0)(V_i),
\end{align*}
where $\widetilde{\ell}_{0}(X)=w^T\big(U-E[U|V]\big)$ is the efficient score function and $\widetilde{I}_0=P\big(U-E[U|V]\big)\big(U-E[U|V]\big)^T=E_{\theta_0,\eta_0}\widetilde{\ell}_{0}\widetilde{\ell}_{0}^T$ the efficient information matrix.

We next analyze four terms in the preceding display. By central limit theorem, the third term is $O_{P_0}\big(\sqrt{n}|\theta-\theta_0|^2\big)$.
An upper bound for the last term could be obtained by applying Lemma~\ref{Lemma:MI} conditioning on $V_i$'s, where the corresponding semimetric $d$ is bounded by the sup-norm. Inequality~\eqref{eq:ce} provides an upper bound for the covering entropy of the space $\{\eta-\eta_0:\eta\in\mathcal{H}_n\}$. Note that even working with the enlarged set $\mathcal{H}_n$ described after assumption (A5), the additional term in the upper bound is negligible. Since $\|\eta-\eta_0\|_n\leq M\rho_n$ for any $\eta\in\mathcal{H}_n$, and $U_i$ conditioning on $V_i$ are bounded and i.i.d.\! with $E\{U_i-E[U|V_i]|V_i\}=0$, an application of Lemma~\ref{Lemma:MI} and inequality~\eqref{eq:ce} yields
\begin{equation}\label{eq:maxin}
\begin{aligned}
    &E_0\bigg\{\sup_{\eta\in H_n}
    \frac{1}{\sqrt{n}}\big|\sum_{i=1}^n\big(U_i-E[U|V_i]\big)
    (\eta-\eta_0)(V_i)\big|\bigg|V_1,\ldots,V_n\bigg\}\\
    \leq K &\, \int_{0}^{M\rho_n}\sqrt{\log N(\epsilon,H_n,\|\cdot\|_{\infty})}\,d\epsilon
    \leq K_4 \sqrt{n}\rho_n^2,
\end{aligned}
\end{equation}
for some constant $K_4$.
Hence
\begin{align*}
   & \sup_{\eta\in H_n }(\theta-\theta_0)^T\big|\sum_{i=1}^n\big(U_i-E[U|V_i]\big)
    (\eta-\eta_0)(V_i)\big|
    =O_{P_0}\big\{n|\theta-\theta_0|\rho_n^2\big\}.
\end{align*}
Combining the above arguments, we verify (A1) with $G_n(t)=\sqrt{n}t^2+n\rho_n^2t$.

\noindent \underline{Verification of (A6)}:
By Lemma 4.3 in \cite{Van2009} and the assumption that each component of $E[U|V=\cdot]$ is at least $\alpha$-smooth, there exists a sequence of functions $\{h_n=(h_{1,n},\ldots,h_{p,n})^T:\bbR^d\to\bbR^p\}$, such that $\|h_n+E[U|V=\cdot]\|_{\infty}\leq C a_n^{-\alpha}\leq\rho_n$ and $\|h_{s,n}\|_{a_n}\leq C a_n^d\leq C\sqrt{n}\rho_n$ for all $s=1,\ldots,p$.
Do a change of variables $\eta\rightarrow\eta+(\theta-\theta_0)^Th_n$. Since for Gaussian processes, the Radon-Nykodym derivative $d\Pi_{H,\cdot+g}/d\Pi_{\mathcal{H}}(W)=\exp(Ug-\|g\|_{a_n}^2/2)$ \citep[Lemma 3.1]{Van2008}, where $U:\mathbb{H}^{a_n}\to \bbR$ is a random operator such that Var$[U(g)]=\|g\|_{a_n}^2$ for any function $g$ in the RKHS $\mathbb{H}^{a_n}$ associated with the GP, we have
\begin{align*}
\log f_n(\eta)&=\log d\Pi_{H,\cdot+(\theta-\theta_0)^Th_n}/d\Pi_{\mathcal{H}}(W)\\
&=(\theta-\theta_0)^TU\big(h_n\big)-(\theta-\theta_0)^TH_n(\theta-\theta_0)/2=O_{P_0}(\bar{G}_n(
|\theta-\theta_0|)),
\end{align*}
with $\bar{G}_n(t)=\sqrt{n}\rho_nt+n\rho^2_nt^2$, where $H_n$ is a $p\times p$ matrix with $H_{st}=\langle h_{s,n},h_{t,n}\rangle_{a_n}\leq Cn\rho_n^2$.

\noindent \underline{Verification of (A4)}: For the same $h_n$ as defined above, we have
\begin{align*}
    \Delta\eta(\theta_n)=-(\theta_n-\theta_0)^T(E[U|V=\cdot]+h_n)=O(|\theta_n-\theta_0|\rho_n).
\end{align*}
Then for $\eta\in\mathcal{H}_n$ and $\theta$ satisfying $|\theta_n-\theta_0|=o_{P_0}(1)$,
\begin{align*}
    &l_n\big(\theta_0,\eta+(\theta_n-\theta_0)^T(h_n-h^\ast)\big)-l_n\big(\theta_0,
    \eta\big)\\
    =&-\frac{1}{2}\sum_{i=1}^n\big(w_i+
    (\eta-\eta_0)(V_i)+(\theta_n-\theta_0)^T(h_n-h^\ast))(V_i)\big)^2
    +\frac{1}{2}\sum_{i=1}^n\big(w_i+
    (\eta-\eta_0)(V_i)\big)^2\\
    =&-(\theta_n-\theta_0)^T\sum_{i=1}^n\big(w_i+(\eta-\eta_0)(V_i)\big)(h_n-h^\ast)(V_i)
   +O\big(|\theta_n-\theta_0|^2\cdot\|h_n-h^\ast\|_n\big).
\end{align*}
By Cauchy's inequality, for $\eta\in\mathcal{H}_n$
\begin{align*}
    \big|\sum_{i=1}^n(\eta-\eta_0)(V_i)(h_n-h^\ast)(V_i)\big|\leq n\|\eta-\eta_0\|_n\|h_n-h^\ast\|_n=O(n\rho_n^2).
\end{align*}
Since $E\big|\sum_{i=1}^nw_i(h_n-h^\ast)(V_i)\big|^2= n\|h_n-h^\ast\|_{n}^2=O(n\rho_n^2)$, we obtain
$\big|\sum_{i=1}^nw_i(h_n-h^\ast)(V_i)\big|=O_{P_0}(\sqrt{n}\rho_n)$.
Combining the above three, we have
\begin{align*}
    l_n\big(\theta_0,\eta+(\theta_n-\theta_0)^T(h_n-h^\ast)\big)-l_n\big(\theta_0,
    \eta\big)=O_{P_0}(\bar{G}_n(|\theta_n-\theta_0|)),
\end{align*}
with $\bar{G}_n(t)=\sqrt{n}\rho_n t+n\rho_n^2t+n\rho_n t^2$.

Finally, applying Theorem \ref{thm:MAIN2} yields the second order semiparametric BvM theorem for the independent prior with a remainder term
\begin{align*}
    G_n(n^{-1/2}\log n)+\bar{G}_n(n^{-1/2}\log n) +\delta_n\sim \sqrt{n}\rho_n^2\log n.
\end{align*}

\paragraph{Dependent prior:}

\underline{Verification of Assumption \ref{A.2}}:
Without loss of generality, we assume that $\|\widehat{h}_n\|_{\infty}\leq 1$.
Similar to the independent prior case, we construct $\mathcal{F}_n$ as $\mathcal{F}_n^{\theta}\oplus\mathcal{F}_n^{\eta}$, with
\begin{align}
    &\mathcal{F}_n^{\theta}=[-c\sqrt{n}, c\sqrt{n}]^p\ \text{ and }\ \mathcal{F}_n^{\eta}=\rho_n \mathcal{C}^{\alpha}_1+M_n\mathbb{H}_1^{a_n}+\big\{(\theta-\theta_0)^T\widehat{h}_n:\, \theta\in\mathcal{F}_n^{\theta} \big\}.\label{eq:dsieves}
\end{align}
Comparing to the sieve~\eqref{eq:isieves} for the independent prior, the third term $\big\{(\theta-\theta_0)^T\widehat{h}_n:\, \theta\in\mathcal{F}_n^{\theta} \big\}$ is added to reflect the dependence structure. For such a sieve $\mathcal{F}_n$, the covering entropy upper bound in inequality~\eqref{eq:ice} is still true for some constant $c$.

Moreover, we have the following complement probability bound on $\mathcal{F}_n$,
\begin{align*}
\Pi(\mathcal{F}_n^c) \leq \Pi_{\Theta}((\mathcal{F}_n^{\theta})^c)+\Pi((\mathcal{F}_n^{\eta})^c).
\end{align*}
The first term above can be bounded by $e^{-c_2n\rho_n^2}$ for some constant $c_2$ and the second term satisfies
\begin{align*}
\Pi((\mathcal{F}_n^{\eta})^c)&\leq \Pi_{\Theta}((\mathcal{F}_n^{\theta})^c)+\int_{\theta\in \mathcal{F}_n^{\theta}} \Pi_{\mathcal{H}}^{\theta}\big(\eta\not\in\rho_n \mathcal{C}^{\alpha}_1+M_n\mathbb{H}_1^{a_n}+(\theta-\theta_0)^T\widehat{h}_n\big)d\Pi_{\Theta}(\theta)\\
&\overset{(i)}{\leq} \exp\{-c_2n\rho_n^2\}+\Pi_{\mathcal{H}}^{\theta_0}(\eta\not\in\rho_n \mathcal{C}^{\alpha}_1+M_n\mathbb{H}_1^{a_n})\overset{(ii)}{\leq} 2\exp\{-c_2n\rho_n^2\}.
\end{align*}
Step (i) follows since by the definition of the dependent prior we have $\Pi^{\theta}_H(\eta\in A)=\Pi^{\theta_0}_H(\eta+(\theta-\theta_0)^T\widehat{h}_n\in A)$ for any measurable subset of $\mathcal{H}$; while step (ii) follows by inequality~\eqref{eq:GPsp2}. By combining the above arguments, we obtain $\Pi(\mathcal{F}_n^c) \leq 3e^{-c_2n\rho_n^2}$.

Finally, there exists a constant, denoted by $c_3$, such that
\begin{align*}
    \Pi(B_n\big(P_0^{(n)},\rho_n)) &\geq \Pi(\|\eta-\eta_0\|_{\infty}\leq \rho_n/2,\,|\theta-\theta_0|\leq \rho_n/4)\\
    &=\int_{|\theta-\theta_0|\leq \rho_n/4} \Pi_{\mathcal{H}}^{\theta}(\|\eta-\eta_0\|_{\infty}\leq \rho_n/2) d\Pi_{\Theta}(\theta)\\
    &\overset{(iii)}{\geq} \Pi_{\mathcal{H}}^{\theta_0}(\|\eta-\eta_0\|_{\infty}\leq \rho_n/4)\cdot \Pi_{\Theta}(|\theta-\theta_0|\leq \rho_n/4)
    \overset{(iv)}{\geq} \exp(-c_3n\rho_n^2),
\end{align*}
where step (iii) follows since by the definition of the dependent prior we have
\begin{align*}
\Pi_{\mathcal{H}}^{\theta}(\|\eta-\eta_0\|_{\infty}\leq \rho_n/2)&=
\Pi_{\mathcal{H}}^{\theta_0}(\|\eta+(\theta-\theta_0)^T\widehat{h}_n-\eta_0\|_{\infty}\leq \rho_n/2)\\
&\geq \Pi_{\mathcal{H}}^{\theta_0}(\|\eta-\eta_0\|_{\infty}\leq \rho_n/4)
\end{align*}
for any $\theta$ satisfying $|\theta-\theta_0|\leq \rho_n/4$, and step (iv) follows
by applying inequality~\eqref{eq:GPsp} with the sup-norm and $\epsilon=\rho_n/4$. Based on these results, the rest of the steps are the same as those for the independent prior.

The verifications of (A1), (A3) and (A4) are also the same as those for the independent prior. Since we do not need to verify (A6) for the dependent prior, an application of Theorem \ref{thm:MAIN2} yields the claimed result.


\subsection{Proof of Theorem \ref{thm:3a'}}
\underline{Verification of Assumption \ref{A.2}}: In this adaptive case, we apply Lemma \ref{le:pcr2} with a modified sieve for the nuisance parameter $\eta$ from \cite {Van2009}. This sieve construction is in the same spirit as the sieve constructed in Theorem \ref{thm:4a,3a} for the non-adaptive scenario.

More specifically, we choose the sieve $\mathcal{F}_n$ as $\mathcal{F}_n^{\theta}\oplus\mathcal{F}_n^{\eta}$, with
\begin{align}
    &\mathcal{F}_n^{\theta}=[-c\sqrt{n}, c\sqrt{n}]^p,\nonumber\\
    \mathcal{F}_n^{\eta}=\bigg(M_n\sqrt{\frac{r_n}{\delta_n}}\mathbb{H}_1^{r_n}&+\rho_n \mathcal{C}^{\alpha}_1\bigg)\cup
    \big(\bigcup_{a\leq\delta_n}(M_n\mathbb{H}_1^a)+\rho_n \mathcal{C}^{\alpha}_1\big),\label{eq:sieves}
\end{align}
with $c$ a sufficiently large constant, $\rho_n=n^{-\alpha/(2\alpha+d)}(\log n)^{d+1}$, and $(M_n,r_n,\delta_n)$ satisfies
\begin{align*}
    &D_2r_n^d\geq 2C_0n\rho_n^2,\quad r_n^{p-d+1}\leq e^{C_0n\rho_n^2},\\
    &M_n^2\geq  8C_0n\rho_n^2,\quad\delta_n=C_1\rho_n/(2\sqrt{d}M).
\end{align*}
Borrowing the results in the proof of Theorem 3.1 in \cite{Van2009} and the intermediate results in the proof of Theorem \ref{thm:4a,3a} about the covering entropy and complementary probability for $M\mathbb{H}_1^a+\epsilon C_1^{\alpha}$ for a fixed bandwidth parameter $a$, we can verify that $\mathcal{F}_n$ satisfies condition\,a and condition\,b in  Lemma \ref{le:pcr2}:
\begin{align}\label{eq:ce}
    \log N(4\epsilon,\mathcal{F}_n,\|\cdot\|_{\infty})\leq K_2 n\rho_n^2(\log n)^{-(d+1)}
    \bigg(\log\bigg(\frac{n}{\epsilon}\bigg)\bigg)^{1+d}
    +K_2\bigg(\frac{\rho_n}{\epsilon}\bigg)^{d/\alpha}+c\log\bigg(\frac{n}{\epsilon}\bigg),
\end{align}
for some constant $K_2$ and for some constant $c_2$,
\begin{align}\label{eq:cp}
    P(\mathcal{F}_n^c)\leq\exp(-c_2n\rho_n^2).
\end{align}
Under the conditions stated in the theorem, the rest of the proof of $\Pi\big(|\theta-\theta_0|\leq M\rho_n, \|\eta-\eta_0\|_n\leq M\rho_n\big|X^{(n)}\big)=1-O_{P_0}(\delta_n)$ is similar to that in Theorem \ref{thm:4a,3a} and is skipped here.

The verifications of (A1) and (A3) are the same as those in the proof of Theorem \ref{thm:4a,3a}.

\noindent \underline{Verification of (A6)}:
We choose $h_n\equiv h^\ast=-E[U|V=\cdot]$, with which (A4) is trivially satisfied.
Since by Lemma 1 in \cite{Ghosal2007}, we have $\Pi(A_n|X^{(n)})=1-O_{P_0}(\delta_n)$ with $A_n=\{A\leq Cn\rho_n^2\}$ for $C$ sufficiently large, where $A$ is the random inverse bandwidth parameter in the GP prior. Consequently, we can always assume $A\leq Cn\rho_n^2$ by conditioning on the event $A_n$.
By the assumption on the least favorable direction $h^\ast=-E[U|V=\cdot]$, each component of $E[U|V]$ satisfies $E[U_s|V]\in\mathbb{H}^{t_0}$ for $s=1,\ldots,p$. Then, by Lemma 4.7 in \cite{Van2009}, we have $\|E[U_s|V=\cdot]\|_{a}\leq C_1\sqrt{a}$, where $C_1=\sup_{s}\|E[U_s|V=\cdot]\|_{t_0}/\sqrt{t_0}$ is a constant independent with $a$. Here we recall that the $\|\cdot\|_a$-norm is the norm of the RKHS associated with the kernel $K^{a}$.  Denote the conditional prior of $\eta$ given $(A=a)$ by $\Pi^{a}$. Do a change of variables $\eta\rightarrow\eta-(\theta-\theta_0)^TE[U|V]$. Similar to the proof in Theorem \ref{thm:4a,3a}, the Radon-Nykodym derivative $d\Pi^a_{\cdot+(\theta-\theta_0)^Th^\ast}/d\Pi^a_{\cdot}(W)$ takes a form as
$\exp((\theta-\theta_0)^TUh^\ast-(\theta-\theta_0)^TH(\theta-\theta_0)/2)$ where $H$ is a $p\times p$ matrix with $H_{st}=\langle h^\ast_s,h^\ast_t\rangle_{a}\leq C_1^2a$ and Var$\{UE[U_j|V]\}=\|E[U_j|V=\cdot]\|_{a}^2$ for $j=1,\ldots,p$. Therefore, we obtain
\begin{align*}
\log f_n(\eta)&=\log d\Pi^a_{\cdot+(\theta-\theta_0)^Th^\ast}/d\Pi^a_{\cdot}(W)\\
&=(\theta-\theta_0)^TU\big(E[U|V]\big)-(\theta-\theta_0)^TH(\theta-\theta_0)/2
=O_{P_0}(\bar{G}_n(|\theta-\theta_0|)),
\end{align*}
with $\bar{G}_n(t)=\sqrt{n}\rho_nt+n\rho^2_nt^2$.

Finally, applying Theorem \ref{thm:MAIN2} yields the second order semiparametric BvM theorem with a remainder term
\begin{align*}
    G_n(n^{-1/2}\log n)+\bar{G}_n(n^{-1/2}\log n) +\delta_n\sim n^{1/2}\rho_n^2\log n.
\end{align*}


\subsection{Proof of Lemma~\ref{LemmaCM}}\label{sec:lem}
The proof is based on Lemma~\ref{Lemma:MI}. Since $U_i$'s are bounded, conditioning on $V_i$'s, we have that $W_{\theta,\eta}:\,=\frac{1}{\sqrt{n}}\sum_{i=1}^n \big(U_i-E(U_i|V_i)\big)\cdot\big((\eta-\eta_0)(V_i)+(\theta-\theta_0)^TE(U_i|V_i)\big)$ is a sub-Gaussian process indexed by $(\theta,\eta)\in \mathcal{F}_n=\mathcal{F}_n^{\theta}\times \mathcal{F}_n^{\eta}$ with the semimetric $d$ (defined in Lemma~\ref{Lemma:MI}) given by $d((\theta,\eta),\, (\theta',\eta'))=\big[n^{-1}\sum_{i=1}^n \big((\eta-\eta')(V_i)+(\theta-\theta')^TE[U_i|V_i]\big)^2\big]^{1/2}$, which is dominated by $2\|\eta-\eta'\|_{\infty}+2|\theta-\theta'|$.
Then by applying inequality~\eqref{eq:ice} and noticing the assumption that $\alpha>d/2$, we have that for any $\delta>\rho_n$,
\begin{align}\label{eqnDudley}
\int_{0}^{\delta}\sqrt{1+\log N(\epsilon,\mathcal{F}_n,d)}\, d\epsilon\leq C \sqrt{n}\rho_n \delta,
\end{align}
for some constant $C$.
By applying Lemma~\ref{Lemma:MI} and inequality~\eqref{eqnDudley} with $\delta=\rho_n$,
we further obtain
\begin{align*}
 \sup_{\theta\in\mathcal{F}_n^{\theta},\,\eta\in\mathcal{F}_n^{\eta}, I_n\leq \rho_n} \bigg|\frac{1}{\sqrt{n}}\sum_{i=1}^n \big(U_i-E(U_i|V_i)\big)\cdot\big((\eta-\eta_0)(V_i)+(\theta-\theta_0)^TE(U_i|V_i)\big)\bigg|=O_{P_0}(\sqrt{n}\rho_n^2).
\end{align*}

To prove the claimed bound, we only need to show that
\begin{align*}
 \sup_{\theta\in\mathcal{F}_n^{\theta},\,\eta\in\mathcal{F}_n^{\eta}, I_n\geq \rho_n} \frac{\big|\frac{1}{\sqrt{n}}\sum_{i=1}^n \big(U_i-E(U_i|V_i)\big)\cdot\big((\eta-\eta_0)(V_i)+(\theta-\theta_0)^TE(U_i|V_i)\big)\big|}{\sqrt{n}\rho_n I_n\log n}=O_{P_0}(1).
\end{align*}
By the reproducing property of RKHS $\mathcal{H}^a$ and our definition of the sieve $\mathcal{F}_n$, $I_n$ is bounded by $L\sqrt{n}$ for some constant $L$.
We will apply the peeling technique by dividing the range $(\rho_n,L\sqrt{n})$ of $I_n$ into $\bigcup_{s=1}^S [\rho_n 2^{s-1}, \rho_n 2^s)$, where $S\leq c \log n$ for some constant $c$.
For each interval $ [\rho_n 2^{s-1}, \rho_n 2^s)$, we first apply Lemma~\ref{Lemma:MI} and inequality~\eqref{eqnDudley} with $\delta=\rho_n 2^s$ and then add them up to obtain
\begin{align*}
&P_0\bigg( \sup_{\theta\in\mathcal{F}_n^{\theta},\,\eta\in\mathcal{F}_n^{\eta}, I_n\geq \rho_n} \frac{\big|\frac{1}{\sqrt{n}}\sum_{i=1}^n \big(U_i-E(U_i|V_i)\big)\cdot\big((\eta-\eta_0)(V_i)+(\theta-\theta_0)^TE(U_i|V_i)\big)\big|}{\sqrt{n}\rho_n I_n}\geq \log n\bigg)\\
\leq &\, \sum_{s=1}^S P_0\bigg(\sup_{\substack{{\theta\in\mathcal{F}_n^{\theta},\,\eta\in\mathcal{F}_n^{\eta},}\\{\rho_n 2^{s-1}\leq  I_n<\rho_n2^{s}}}} \frac{\big|\frac{1}{\sqrt{n}}\sum_{i=1}^n \big(U_i-E(U_i|V_i)\big)\cdot\big((\eta-\eta_0)(V_i)+(\theta-\theta_0)^TE(U_i|V_i)\big)\big|}{\sqrt{n}\rho_n I_n}\geq \log n\bigg)\\
\leq &\, \sum_{s=1}^S P_0\bigg(\sup_{\substack{{\theta\in\mathcal{F}_n^{\theta},\,\eta\in\mathcal{F}_n^{\eta},}\\{\rho_n 2^{s-1}\leq  I_n<\rho_n2^{s}}}} \bigg|\frac{1}{\sqrt{n}}\sum_{i=1}^n \big(U_i-E(U_i|V_i)\big)\cdot\big((\eta-\eta_0)(V_i)+(\theta-\theta_0)^TE(U_i|V_i)\big)\bigg|\\
&\qquad\qquad\qquad\qquad\qquad\qquad\qquad\qquad
\qquad\qquad\qquad\qquad\qquad\qquad\qquad\geq \sqrt{n}\rho_n^2 2^{s-1} \log n\bigg)\\
\leq &\, \sum_{s=1}^S 2\exp\{-c_0(\log n )^2\}\leq 2c\log n\cdot \exp\{-c_0(\log n )^2\} \to 0,\mbox{ as }n\to\infty.
\end{align*}
This completes the proof of the lemma.


\subsection{Proof of Theorem~\ref{thm:4}}
This theorem is proved by combining the verification of Assumption~\ref{A.2} in Theorem~\ref{thm:3a'} and the proof of the dependent prior part in Theorem~\ref{thm:4a,3a} with the sieve $\mathcal{F}_n=\mathcal{F}_n^{\theta}\oplus\mathcal{F}_n^{\eta}$ and localization sequence $\mathcal{H}_n=\{\eta\in\mathcal{F}_n^{\eta}:\|\eta-\eta_0\|_n\leq M\rho_n\}$, where
\begin{align}
    &\qquad\qquad\qquad\mathcal{F}_n^{\theta}=[-c\sqrt{n}, c\sqrt{n}]^p,\nonumber\\
    \mathcal{F}_n^{\eta}=\bigg(M_n\sqrt{\frac{r_n}{\delta_n}}\mathbb{H}_1^{r_n}&+\rho_n \mathcal{C}^{\alpha}_1\bigg)\cup
    \big(\bigcup_{a\leq\delta_n}(M_n\mathbb{H}_1^a)+\rho_n \mathcal{C}^{\alpha}_1\big)+\big\{(\theta-\theta_0)^T\widehat{h}_n:\, \theta\in\mathcal{F}_n^{\theta} \big\}.\label{eq:sieves}
\end{align}
Therefore we omit the proof of this theorem here.


\subsection{Proof of Theorem \ref{thm:5}}

The verification of Assumption \ref{A.2} for the GPLM is similar to that of Theorem~\ref{thm:3a'} with the sieve $\mathcal{F}_n$ given by \eqref{eq:sieves} and localization sequence $\mathcal{H}_n=\{\eta\in\mathcal{F}_n^{\eta}:\|\eta-\eta_0\|_n\leq M\rho_n\}$, and we also have the three inequalities \eqref{eq:ice}, \eqref{eq:icp} and \eqref{eq:icpc} for this $\mathcal{F}_n$. It remains to check whether we can replace the Hellinger metric $d_n$ in Lemma~\ref{le:pcr2} by the empirical metric $\|\cdot\|_n$ for the GPLM, which is indicated by the following lemma. Recall that for semiparametric models, we choose the parameter $\lambda$ in Lemma~\ref{le:pcr2} to be the pair $(\theta,\eta)$ and the empirical distance between the regression function under parameters $\lambda=(\theta,\eta)$ and $\lambda'=(\theta',\eta')$ is given by $\|U^T(\theta-\theta')+\eta-\eta'\|_n= \sqrt{n^{-1}\sum_{i=1}^n\big(U_i^T(\theta-\theta')+\eta(V_i)-\eta'(V_i)\big)^2}$.

The proof of Lemma~\ref{le:pcr4} is provided in the next subsection.
\begin{Lemma}\label{le:pcr4}
For the GPLM, under Assumption~\ref{A.1} and condition a, b and c of Lemma~\ref{le:pcr2}, there exists some constant $C_1>0$ and large enough $M$ such that
\begin{align}
    \Pi\big(\|U^T(\theta-\theta_0)+\eta-\eta_0\|_n\geq M\xi_n\big|X_1,\ldots,X_n\big)&=O_{P_{\lambda_0}^{(n)}}(e^{-C_1n\xi_n^2}),\label{eq:dpcr2}\\
    \Pi\big((\theta,\eta)\not\in \mathcal{F}_n\big|X_1,\ldots,X_n\big)&=O_{P_{\lambda_0}^{(n)}}(e^{-C_1n\xi_n^2}).\label{eq:dpcr3}
\end{align}
\end{Lemma}

Based on Lemma~\ref{le:pcr4} and inequalities \eqref{eq:ice}, \eqref{eq:icp} and \eqref{eq:icpc}, the rest of the steps are the same as those in the proof of Theorem~\ref{thm:4a,3a}.

Next, we prove (A1). By Assumption \ref{A.2}, the posterior of $\eta$ concentrates its mass in a small neighborhood $\mathcal{H}_n=\{\eta\in\mathcal{F}_n^{\eta}:\|\eta-\eta_0\|_n\leq M\rho_n\}$. Write $q_n(\theta,\eta)=\sum_{i=1}^nq_{\theta,\eta}(Y_i)$ and recall that $\Delta\eta(\theta)=\eta^*(\theta)-\eta_0=(\theta-\theta_0)^Th^\ast(V)+O(|\theta-\theta_0|^2).$ The proof of Lemma~\ref{le:LAN2} is provided in Section~\ref{sec:lem2}.
\begin{Lemma}\label{le:LAN2}
Under Assumption \ref{A.1}, we have
\begin{align}
q_n\big(\theta_n,\eta+&\Delta\eta(\theta_n)\big)-q_n(\theta_0,\eta)=
(\theta_n-\theta_0)^T\sum_{i=1}^nW_il_0(T_i)(U_i+h^\ast(V_i))\nonumber\\
&-\frac{1}{2}n(\theta_n-\theta_0)^T\widetilde{I}_{0}(\theta_n-\theta_0)
+O_{P_0}[R_n(\max\{|\theta-\theta_0|,n^{-1/2}\log n\})],\label{eq:slan2}
\end{align}
for every sequence $\{\theta_n\}$ satisfying $\theta_n=\theta_0+O_{P_0}(\rho_n)$ and uniformly for every $\eta\in\mathcal{H}_n$, with
$\widetilde{I}_{0}=E_0\big[l_0(T)f_0(T)(U+h^\ast(V))(U+h^\ast(V))^T\big]$
and $R_n(t)=nt^3+\sqrt{n}t^2+n\rho_n t^2+n\rho_n^2t+\sqrt{n}\rho_n^2$.
\end{Lemma}

To apply Theorem \ref{thm:MAIN2}, it remains to verify (A4). By Lemma~\ref{le:glfc}, we have
$\eta-\Delta(\theta_n)+(\theta_n-\theta)^T\widehat{h}_n = O(|\theta-\theta_0|^2 + |\theta-\theta_0| \rho_n)$. Then (A4) is an easy consequence of \eqref{eqn:gplmlr}, \eqref{eq:tay1} and \eqref{eq:tay2} with $\theta=\theta_0$, $\xi_1=\eta$ and $\xi_2=\eta-\Delta(\theta_n)+(\theta_n-\theta)^T\widehat{h}_n$ in the proof of Lemma~\ref{le:LAN2}.


\subsection{Proof of Lemma \ref{le:pcr4}}
According to Section 7.7 in \cite{Ghosal2007}, we only need to verify that there exists a test function $\phi_n:\bbR^n\to \bbR$ for testing $\lambda_0=(\theta_0,\eta_0)$ versus $\lambda_1=(\theta_1,\eta_1)$ relative to the empirical norm $\|\lambda_0-\lambda_1\|_n:\,=\|U^T(\theta_0-\theta_1)+\eta_0-\eta_1\|_n$ (instead of $d_n$) that satisfies the conclusion of Lemma 2 in \cite{Ghosal2007}, i.e. $\phi_n$ satisfies
\begin{align}\label{eq:tgoal}
P_{\lambda_0}^{(n)} \phi_n(Y^n) \leq e^{-cn \|\lambda_0-\lambda_1\|_n^2},\quad
\mbox{and}\quad P_{\lambda}^{(n)} (1-\phi_n(Y^n)) \leq e^{-cn \|\lambda_0-\lambda_1\|_n^2}
\end{align}
for all $\lambda$ such that $\|\lambda-\lambda_1\|_n\leq c'\|\lambda_0-\lambda_1\|_n$, where $(c,c')$ are constants independent of $n$ and we use the shorthand $Y^n$ to denote the response vector $(Y_1,\ldots,Y_n)$.

More specifically, we choose
\begin{align}
\phi_n(Y^n)=I\big(\|Y-m_{\theta_0,\eta_0}(T)\|_n^2-\|Y-m_{\theta_1,\eta_1}(T)\|_n^2\geq 0\big),
\end{align}
where $\|Y-m_{\theta,\eta}(T)\|_n^2:\,= n^{-1}\sum_{i=1}^n\big(Y_i-m_{\theta,\eta}(T_i)\big)^2$. Recall that by Assumption~\ref{A.1}, under $P_{\lambda_0}^{(n)}$ the residuals $W_i=Y_i-m_{\theta_0,\eta_0}(T_i)$ are i.i.d.\! sub-Gaussian. Therefore, we have that for any $t>0$,
\begin{align*}
P_{\lambda_0}^{(n)} \phi_n(Y^n)& = P_{\lambda_0}^{(n)}\Big\{\frac{1}{n}\sum_{i=1}^nW_i^2-\frac{1}{n}\sum_{i=1}^n\big(W_i+m_{\theta_0,\eta_0}(T_i)-m_{\theta_1,\eta_1}(T_i)\big)^2 \geq 0\Big\}\\
& =  P_{\lambda_0}^{(n)}\Big\{ \sum_{i=1}^n tW_i \big(m_{\theta_1,\eta_1}(T_i)-m_{\theta_0,\eta_0}(T_i)\big)\geq \frac{t}{2} \sum_{i=1}^n \big(m_{\theta_1,\eta_1}(T_i)-m_{\theta_0,\eta_0}(T_i)\big)^2\Big\}\\
& \overset{(i)}{\leq} \exp\{-\frac{t}{2}\sum_{i=1}^n (m_{\theta_1,\eta_1}(T_i)-m_{\theta_0,\eta_0}(T_i))^2\} \prod_{i=1}^nE_{\lambda_0}\big(e^{t(m_{\theta_1,\eta_1}(T_i)-m_{\theta_0,\eta_0}(T_i)) W_i}\big|T_i\big)
\end{align*}
where in step (i) we have applied Markov's inequality and used the independence among $Y_i$'s. Then by Assumption~\ref{A.1}\! (a), there exists some constant $C_2$ such that
\begin{align*}
P_{\lambda_0}^{(n)} \phi_n(Y^n)&\leq e^{-\frac{t}{2}\sum_{i=1}^n (m_{\theta_1,\eta_1}(T_i)-m_{\theta_0,\eta_0}(T_i))^2} \prod_{i=1}^n e^{Ct^2(m_{\theta_1,\eta_1}(T_i)-m_{\theta_0,\eta_0}(T_i))^2}\\
& = \exp\{-\big(\frac{t}{2}-Ct^2\big)\sum_{i=1}^n (m_{\theta_1,\eta_1}(T_i)-m_{\theta_0,\eta_0}(T_i))^2\}.
\end{align*}
By choosing $t=C^{-1}$ in the above, we obtain
\begin{align*}
P_{\lambda_0}^{(n)} \phi_n(Y^n)\leq  \exp\{-\frac{1}{2C}\sum_{i=1}^n (m_{\theta_1,\eta_1}(T_i)-m_{\theta_0,\eta_0}(T_i))^2\}.
\end{align*}
According to Assumption~\ref{A.1}\! (b), we know that $\big(m_{\theta_1,\eta_1}(T_i)-m_{\theta_0,\eta_0}(T_i)\big)^2\geq 2(C_1C_2)^{-1} \big(U_i^T(\theta_1-\theta_0)+(\eta_1-\eta_0)\big)^2$ for $i=1,\ldots,n$, implying
\begin{align*}
P_{\lambda_0}^{(n)} \phi_n(Y^n)\leq \exp\{-cn \|\lambda_0-\lambda_1\|_n^2\},
\end{align*}
where the constant $c=(CC_1C_2)^{-1}$. This proves the first part of \eqref{eq:tgoal}.

Now we prove the second part of \eqref{eq:tgoal}. By Assumption~\ref{A.1}, under $P_{\lambda}^{(n)}$ the residuals $W_i=Y_i-m_{\theta,\eta}(T_i)$ are i.i.d.\! sub-Gaussian. Consequently, for any $\lambda$ and any $t>0$ we have
\begin{align*}
&P_{\lambda}^{(n)} \big(1-\phi_n(Y^n)\big)\\
 =&\,  P_{\lambda}^{(n)}\Big\{\frac{1}{n}\sum_{i=1}^n\big(W_i+m_{\theta,\eta}(T_i)-m_{\theta_1,\eta_1}(T_i)\big)^2-\frac{1}{n}\sum_{i=1}^n\big(W_i+m_{\theta,\eta}(T_i)-m_{\theta_0,\eta_0}(T_i)\big)^2 \geq 0\Big\}\\
 =&\,  P_{\lambda}^{(n)}\Big\{ \sum_{i=1}^n tW_i \big(m_{\theta_0,\eta_0}(T_i)-m_{\theta_1,\eta_1}(T_i)\big)\geq \frac{t}{2}\sum_{i=1}^n  \big(\big(m_{\theta,\eta}(T_i)-m_{\theta_0,\eta_0}(T_i)\big)^2-\big(m_{\theta,\eta}(T_i)-m_{\theta_1,\eta_1}(T_i)\big)^2\Big\}.
\end{align*}
Then similar to the first part, by using Assumption~\ref{A.1}\! (b) and applying Markov's inequality we can obtain
\begin{align*}
P_{\lambda}^{(n)} \big(1-\phi_n(Y^n)\big)\leq  \exp\{-D_1tn \|\lambda-\lambda_0\|_n^2+D_2tn\|\lambda-\lambda_1\|_n^2+D_3t^2\|\lambda_0-\lambda_1\|_n^2\},
\end{align*}
for some constants $D_1$, $D_2$ and $D_3$. Therefore, for any $\lambda$
such that $\|\lambda-\lambda_1\|_n\leq c'\|\lambda_0-\lambda_1\|_n$, where $c'=\sqrt{D_1}/(\sqrt{D_1}+\sqrt{2D_2})$, we have
$\|\lambda-\lambda_0\|_n\geq (1-c')\|\lambda_0-\lambda_1\|_n$ and
\begin{align*}
P_{\lambda}^{(n)} \big(1-\phi_n(Y^n)\big)\leq  \exp\{-D_4tn \|\lambda_0-\lambda_1\|^2_n+D_3t^2\|\lambda_0-\lambda_1\|_n^2\},
\end{align*}
where $D_4=D_1D_2/(\sqrt{D_1}+\sqrt{2D_2})^2$. Finally, by choosing $t=D_4/(2D-3)$ in the above, we obtain
\begin{align*}
P_{\lambda}^{(n)} \big(1-\phi_n(Y^n)\big)\leq  \exp\{-ct^2\|\lambda_0-\lambda_1\|_n^2\},
\end{align*}
where $c=D_4^2/(2D_3)$. This proves the second part of \eqref{eq:tgoal}.


\subsection{Proof of Lemma \ref{le:LAN2}}\label{sec:lem2}
By the definitions of $q_n$ and $q_{\theta,\eta}$, we get
\begin{align}
    q_n\big(\theta,\eta+\Delta\eta(\theta)\big)-q_n(\theta_0,\eta)
    =&\sum_{i=1}^nW_i
    \int_{m_{\theta_0,\eta}(T_i)}^{m_{\theta,\eta+\Delta\eta(\theta)}(T_i)}\frac{1}{V(s)}ds\notag \\
    &-\sum_{i=1}^n\int_{m_{\theta_0,\eta}(T_i)}^{m_{\theta,\eta+\Delta\eta(\theta)}(T_i)}
    \frac{(s-m_0(T_i))}{V(s)}ds\triangleq I-I\!I, \label{eqn:gplmlr}
\end{align}
with $W_i=Y_i-m_0(T_i)$ satisfying $E_0W_i=0$ and Assumption \ref{A.1}(\ref{assump:A1:a}).

By applying Taylor expansion and Assumption \ref{A.1}(\ref{assump:A1:b}), we have for any $\xi_0,\xi_1,\xi_2\in\bbR$,
\begin{align}
 \int_{F(\xi_1)}^{F(\xi_2)}\frac{1}{V(s)}ds
 =&l(\xi_1)(\xi_2-\xi_1)+e_1(\xi_0)(\xi_2-\xi_1)^2+O((\xi_2-\xi_1)^3)\nonumber\\
 =&l(\xi_0)(\xi_2-\xi_1)+e_1(\xi_0)(\xi_2-\xi_1)^2+e_2(\xi_0)(\xi_2-\xi_1)(\xi_1-\xi_0)\nonumber\\
 &+O\big\{(\xi_2-\xi_1)^3+(\xi_2-\xi_1)^2(\xi_1-\xi_0)+(\xi_2-\xi_1)(\xi_1-\xi_0)^2\big\},\label{eq:tay1}\\
 \int_{F(\xi_1)}^{F(\xi_2)}\frac{s-F(\xi_0)}{V(s)}ds
 =&l(\xi_1)\big(F(\xi_1)-F(\xi_0)\big)(\xi_2-\xi_1)+\frac{1}{2}l(\xi_1)f(\xi_1)(\xi_2-\xi_1)^2\nonumber\\
 &+O\big\{(\xi_2-\xi_1)^3+(\xi_2-\xi_1)^2(\xi_1-\xi_0)\big\}\nonumber\\
 =&l(\xi_0)f(\xi_0)(\xi_2-\xi_1)(\xi_1-\xi_0)+\frac{1}{2}l(\xi_0)f(\xi_0)(\xi_2-\xi_1)^2\nonumber\\
 &+O\big\{(\xi_2-\xi_1)^3+(\xi_2-\xi_1)^2(\xi_1-\xi_0)+(\xi_2-\xi_1)(\xi_1-\xi_0)^2\big\},\label{eq:tay2}
\end{align}
with $e_1(\xi)$ and $e_2(\xi)$ fixed bounded functions.

By the definitions of $g_{\theta,\eta}$ and $\Delta\eta(\theta)$, we have
\begin{align*}
    g_{\theta_0,\eta}(T)-g_0(T)
    =&(\eta-\eta_0)(V),\\
    g_{\theta,\eta+\Delta\eta(\theta)}(T)-g_{\theta_0,\eta}(T)
    =&(\theta-\theta_0)^Th_1(T)+O(|\theta-\theta_0|^2),
\end{align*}
with $h_1(T)=U+h^\ast(V)$.
Combining the above and the definition of $l_0$, $f_0$ and $m_{\theta,\eta}$, and \eqref{eq:tay1} with $\xi_0=g_0$, $\xi_1=g_{\theta_0,\eta}$, and $\xi_2=g_{\theta,\eta+\Delta\eta(\theta)}$, we get
\begin{align*}
    I=(\theta-\theta_0)^T\sum_{i=1}^n&W_il_0(T_i)h_1(T_i)\\
    &+(\theta-\theta_0)^T\sum_{i=1}^nW_ie_2(g_0(T_i))h_1(T_i)(\eta-\eta_0)(V_i)
    +O_{P_0}(\sqrt{n}|\theta-\theta_0|^2),
\end{align*}
where the last term is obtained by combining central limit theorem and the fact that $E_0W_i=0$ and $E_0W_i^2<\infty$.

Recall that $e_2$ is some bounded function in the expansion~\eqref{eq:tay1}.
Since $W_i$ is sub-Gaussian, $e_2$ and $h_1$ are bounded functions, we have that $W_ie_2(g_0(T_i))h_1(T_i)$ is also sub-Gaussian. Moreover, since $E_0W_ie_2(g_0(T_i))h_1(T_i)=E_0[e_2(g_0(T_i))h_1(T_i)E_0(W_i|T_i)]=0$, similar to \eqref{eq:maxin}, by applying Lemma~\ref{Lemma:MI}, we get
\begin{align*}
    &E_0\bigg\{\sup_{\eta\in H_n}
    \frac{1}{\sqrt{n}}\big|\sum_{i=1}^nW_ie_2(g_0(T_i))h_1(T_i)
    (\eta-\eta_0)(V_i)\big|\bigg|V_1,\ldots,V_n\bigg\}\\
    \lesssim &\,\int_{0}^{\rho_n}\sqrt{1+\log N(\epsilon,H_n,\|\cdot\|_{\infty})}d\epsilon\\
    \lesssim &\,\sqrt{n}\rho_n^2+\rho_n\asymp \sqrt{n}\rho_n^2.
\end{align*}
Combining the above two, we get
\begin{align}\label{eq:exI}
    I=(\theta-\theta_0)^T\sum_{i=1}^nW_il_0(T_i)h_1(T_i)
    +O_{P_0}\big\{\sqrt{n}|\theta-\theta_0|^2+n|\theta-\theta_0|\rho_n^2\big\}.
\end{align}

Similarly, using \eqref{eq:tay2} and the same choices for $\xi_0$, $\xi_1$ and $\xi_2$, we get
\begin{align*}
    I\!I=&(\theta-\theta_0)^T\sum_{i=1}^nl_0(T_i)f_0(T_i)(U_i+h^\ast(V_i))(\eta-\eta_0)(V_i)\\
    &+\frac{1}{2}\sum_{i=1}^nl_0(T_i)f_0(T_i)\big((\theta-\theta_0)^Th_2(T_i)\big)^2
    +O_{P_0}\big\{n|\theta-\theta_0|^3+n|\theta-\theta_0|^2\rho_n+n|\theta-\theta_0|\rho_n^2\big\},
\end{align*}
where $h_2(t)=u-E[U|V=v]$.
By definition of $h^\ast$, we have
\begin{align*}
    &E_0\big[l_0(T_i)f_0(T_i)(U_i+h^\ast(V_i))(\eta-\eta_0)(V_i)\big]\\
    =&E_0\big[(\eta-\eta_0)(V_i)E_0(l_0(T_i)f_0(T_i)(U_i+h^\ast(V_i))|V_i)\big]=0.
\end{align*}
Therefore, by applying Lemma~\ref{Lemma:MI}, we get
\begin{align*}
    &E_0\bigg\{\sup_{\eta\in H_n}
    \frac{1}{\sqrt{n}}\big|\sum_{i=1}^nl_0(T_i)f_0(T_i)(U_i+h^\ast(V_i))(\eta-\eta_0)(V_i)\big|\bigg|V_1,\ldots,V_n\bigg\}\\
    \lesssim &\,\int_{0}^{\rho_n}\sqrt{1+\log N(\epsilon,H_n,\|\cdot\|_{\infty})}d\epsilon
    \lesssim \sqrt{n}\rho_n^2+\rho_n\asymp \sqrt{n}\rho_n^2.
\end{align*}
By central limit theorem, we have
\begin{align*}
    &\frac{1}{2}\sum_{i=1}^nl_0(T_i)f_0(T_i)\big((\theta-\theta_0)^Th_2(T_i)\big)^2\\
    =&\frac{n}{2}(\theta-\theta_0)^TE_0\big[l_0(T)f_0(T)h_1(T)(h_1(T))^T\big](\theta-\theta_0)+O_{P_0}(\sqrt{n}|\theta-\theta_0|^2).
\end{align*}

Combining the above, we have
\begin{align}
    I\!I=&\frac{n}{2}(\theta-\theta_0)^TE_0\big[l_0(T)f_0(T)h_1(T)(h_1(T))^T\big](\theta-\theta_0)\nonumber\\
    &+O_{P_0}\big\{n|\theta-\theta_0|^3+\sqrt{n}|\theta-\theta_0|^2
    +n|\theta-\theta_0|^2\rho_n+n|\theta-\theta_0|\rho_n^2+\sqrt{n}\rho_n^2\big\}.\label{eq:exII}
\end{align}
By \eqref{eq:exI} and \eqref{eq:exII}, the lemma is proved.

\subsection{Proof of Theorem \ref{thm:6}}

\underline{Verification of Assumption \ref{A.2}}: We apply Lemma \ref{le:pcr2}. We first state the following lemma for the Cox model with current status data, whose proof is similar to those of Lemma 7 and Lemma 8 of \cite{Castillo2012} for the Cox proportional hazards model and is omitted here.

\begin{Lemma}
Let $\Lambda_1$ and $\Lambda_2$ be cumulative hazard functions associated with
baseline hazard functions $r_1$ and $r_2$. Let $h$ be the Hellinger metric. If $\|r_1-r_2\|_{\infty}+|\theta_1-\theta_2|\lesssim 1$, then
\begin{align*}
    &h^2(p_{\theta_1,\Lambda_1},p_{\theta_2,\Lambda_2})\lesssim \|r_1-r_2\|^2_{\infty}+|\theta_1-\theta_2|^2,\\
    &K(p_{\theta_1,\Lambda_1},p_{\theta_2,\Lambda_2})\lesssim \|r_1-r_2\|^2_{\infty}+|\theta_1-\theta_2|^2,\\
    &V_{0,k}(p_{\theta_1,\Lambda_1},p_{\theta_2,\Lambda_2})\lesssim \|r_1-r_2\|^k_{\infty}+|\theta_1-\theta_2|^k.
\end{align*}
\end{Lemma}

Let $\rho_n=n^{-1/3}$. According to this lemma, we have $\{(\theta,r):|\theta-\theta_0|\leq \rho_n, \|r-r_0\|_{\infty}\leq \rho_n\}\subset B_n(P_0^{(n)},\rho_n;k)$. For the Riemann-Liouville prior \eqref{eq:RLprior}, it can be shown \citep[Lemma 16]{Castillo2012} that for any Lipschitz continuous function $r_0$ and $\epsilon>0$,
\begin{align*}
    \Pi(\|r-r_0\|_{\infty}\leq\epsilon)\gtrsim \exp(-\epsilon^{-1}).
\end{align*}
Let $\mathcal{C}_1$ be the unit ball of $\mathcal{C}[0,\tau]$ and $\mathbb{H}_1$ be the unit RKHS associated with the Gaussian process given in \eqref{eq:RLprior}. Define $\mathcal{F}_n^{\Lambda}=\rho_n \mathcal{C}_1+\sqrt{10Cn}\rho_n\mathbb{H}_1$ for some large enough $C$. Then according to Theorem 2.1 of \cite{Van2008b},
\begin{align*}
    \log N(3\rho_n, \mathcal{F}_n^{\Lambda}, \|\cdot\|_{\infty})&\leq 6C n\rho_n^2,\\
    \Pi(r\notin \mathcal{F}_n^{\Lambda})&\leq\exp(-Cn\rho_n^2).
\end{align*}
Now we choose the sieve $\mathcal{F}_n$ as $\mathcal{F}_n^{\theta}\oplus\mathcal{F}_n^{\Lambda}$ with the above $\mathcal{F}_n^{\Lambda}$ and $\mathcal{F}_n^{\theta}=[-C\sqrt{n},C\sqrt{n}]^p$. Then the three conditions in Lemma \ref{le:pcr2} are satisfied. By Lemma 25.85 of \cite{Van1998}, we have $d_n^2\big((\theta,\Lambda),(\theta_0,\Lambda_0)\big)\gtrsim \|\Lambda-\Lambda_0\|_n^2+|\theta-\theta_0|^2$. Therefore, Lemma \ref{le:pcr2} implies that $\Pi(\|\Lambda-\Lambda_0\|_n\geq\rho_n,|\theta-\theta_0|\geq\rho_n|X_1,\ldots,X_n)=O_{P_0}(e^{-C_1n\rho_n^2})$.

In the rest of the proof, we set $\mathcal{H}_n=\big\{\Lambda: \Lambda\text{ is nondecreasing, } \|\Lambda-\Lambda_0\|_n\leq\rho_n, \|\Lambda\|_{\infty}\leq C\}$ for some large enough $C$. Then by the above statements, we have $\Pi(\mathcal{H}_n|X_1,\ldots,X_n)=1-O_{P_0}(e^{-C_1n\rho_n^2})$.

\underline{Verification of (A1)}:
By applying Taylor's expansions on $\theta$ and central limit theorem, the log likelihood difference in (A1) can be written as
\begin{equation}\label{eq:lexp}
\begin{aligned}
    l_n(\theta,\Lambda+\Delta\Lambda(\theta))-l_n(\theta_0,\Lambda)=&(\theta-\theta_0)^T
    \sum_{i=1}^n\big(Z_i\Lambda(C_i)+h^\ast(C_i)\big)Q_{\theta_0,\Lambda}(X_i)\\
    &-\frac{1}{2}n(\theta-\theta_0)^T\widetilde{I}_{\theta_0,\Lambda_0}(\theta-\theta_0)+O_{P_0}\big(n|\theta-\theta_0|^3
    +n|\theta-\theta_0|^2\rho_n\big),
\end{aligned}
\end{equation}
where $\widetilde{I}_{\theta_0,\Lambda_0}=P\big\{Q_{\theta_0,\Lambda_0}^2(X)(Z\Lambda_0(C)+h^\ast(C))^T(Z\Lambda_0(C)+h^\ast(C))\big\}$ is the efficient information matrix.

Now we focus on the first term on the RHS of (\ref{eq:lexp}). Let $g_{\Lambda}(x)=(z\Lambda(c)+h^\ast(c))Q_{\theta_0,\Lambda}(x)-E_0\big[(Z\Lambda(C)+h^\ast(C))Q_{\theta_0,\Lambda}(X)\big|C=c\big]$. Then $E_0(g_{\Lambda}(X)|C)=0$ for any $\Lambda\in\mathcal{H}_n$ and $g_{\Lambda}(x)$ is a continuous function of $\Lambda(c)$ with a bounded Lipschitz constant. Therefore, the $\epsilon$-covering entropy of $\{g_{\Lambda}:\Lambda\in \mathcal{H}_n\}$ with respect to the conditional $L_2$-norm conditioning on $\{C_i\}$ is of the same order as $\log N(\epsilon, \mathcal{H}_n,\|\cdot\|_n)$, which is of order $\epsilon^{-1}$ since the functions in $\mathcal{H}_n$ are bounded and nondecreasing \citep[Example 19.11]{Van1998}.
By applying Lemma~\ref{Lemma:MI} conditioning on $C_i$, we get
\begin{align*}
    &E_0\bigg\{\sup_{\Lambda_n\in H_n}
    \big|\mathbb{G}_n\big[(Z\Lambda(C)+h^\ast(C))Q_{\theta_0,\Lambda}(X)\big]-
    \mathbb{G}_n\big[(Z\Lambda_0(C)+h^\ast(C))Q_{\theta_0,\Lambda_0}(X)\big]\bigg|C_1,\ldots,C_n\bigg\}\\
    \lesssim &\int_{0}^{\rho_n}\sqrt{1+\log N(\epsilon,H_n,\|\cdot\|_{n})}d\epsilon \ \lesssim \sqrt{\rho_n} \asymp \sqrt{n}\rho_n^2,
\end{align*}
where the last step follows since $\rho_n= n^{-1/3}$.

As a result of the preceding display, we have
\begin{equation}\label{eq:lexpa}
\begin{aligned}
    &\sum_{i=1}^n\big(Z_i\Lambda(C_i)+h^\ast(C_i)\big)Q_{\theta_0,\Lambda}(X_i)-
    \sum_{i=1}^n\big(Z_i\Lambda_0(C_i)+h^\ast(C_i)\big)Q_{\theta_0,\Lambda_0}(X_i)\\
    =& O_{P_0}(n\rho_n^2)+\sum_{i=1}^nE_0\big[(Z_i\Lambda(C_i)+h^\ast(C_i))Q_{\theta_0,\Lambda}(X_i)-
    (Z_i\Lambda_0(C_i)+h^\ast(C_i))Q_{\theta_0,\Lambda_0}(X_i)\big|C_i\big].
\end{aligned}
\end{equation}
Note that $\widetilde{l}_{\theta,\Lambda}(X)=(Z\Lambda(C)+h^\ast_{\Lambda}(C))Q_{\theta,\Lambda}$ is the efficient score function for the Cox model with current status data \citep[Section 25.11.1]{Van1998}, where $h^\ast_{\Lambda}$ generalizes $h^\ast$ by substituting $\Lambda_0$ with $\Lambda$ in \eqref{eq:lfd2}. By equation (25.60) on P.396 of \cite{Van1998}, $P_{\theta_0,\Lambda_0}\widetilde{l}_{\theta_0,\Lambda}=O_{P_0}(\|\Lambda-\Lambda_0\|_n^2)$. Therefore, for $\Lambda\in\mathcal{H}_n$,
\begin{align}
    &\sum_{i=1}^nE_0\big[(Z_i\Lambda(C_i)+h^\ast(C_i))Q_{\theta_0,\Lambda}(X_i)-
    (Z_i\Lambda_0(C_i)+h^\ast(C_i))Q_{\theta_0,\Lambda_0}(X_i)\big|C_i\big]\nonumber\\
    =&O_{P_0}(n\rho_n^2)+\sum_{i=1}^n(h_{\Lambda_0}^\ast(C_i)-h^\ast_{\Lambda}(C_i))
    E_0\big(Q_{\theta_0,\Lambda}(X_i)-Q_{\theta_0,\Lambda_0}(X_i)\big|C_i\big)\nonumber\\
    =&O_{P_0}(n\rho_n^2)+O_{P_0}(n\|\Lambda-\Lambda_0\|_n^2)=O_{P_0}(n\rho_n^2), \label{eq:lexpb}
\end{align}
since $E_0(Q_{\theta_0,\Lambda_0}(X_i)|C_i)=0$ and $Q_{\theta_0,\Lambda}(x)$ is a continuous function of $\Lambda(C_i)$ with a bounded Lipschitz constant.

Combining \eqref{eq:lexp}, \eqref{eq:lexpa} and \eqref{eq:lexpb}, we obtain
\begin{align*}
    l_n(\theta,\Lambda+\Delta\Lambda(\theta))-l_n(\theta_0,\Lambda)=&(\theta-\theta_0)^T
    \sum_{i=1}^n\widetilde{l}_{\theta_0,\Lambda_0}(X_i)
    -\frac{1}{2}n(\theta-\theta_0)^T\widetilde{I}_{\theta_0,\Lambda_0}(\theta-\theta_0)\\
    &+O_{P_0}\big(n|\theta-\theta_0|^3
    +n|\theta-\theta_0|^2\rho_n+\sqrt{n}\rho_n^2+n|\theta-\theta_0|\rho_n^2\big).
\end{align*}

The verifications of A4 are similar to those in the proof of Theorem \ref{thm:4a,3a} since by the assumption of the theorem, $h^\ast$ is Lipschitz continuous. The verification of A6 is a simple version of the verification of (A1) and is omitted.
Finally, Theorem \ref{thm:6} can be proved by applying Theorem \ref{thm:MAIN2}.

\bibliography{draft}
\end{document}